\documentclass{amsart}
\usepackage{enumitem}
\usepackage{graphicx}	
\usepackage{amsmath, amsthm, amssymb,mathtools}
\usepackage{amsfonts}
\usepackage{hyperref}
\usepackage[dvipsnames]{xcolor}
\usepackage{mathrsfs}

\usepackage[capitalise,noabbrev]{cleveref}

\usepackage[linesnumbered,ruled,vlined]{algorithm2e} 

\newcommand{\algname}[1]{\texttt{#1}}
\numberwithin{equation}{section}

\theoremstyle{plain}
\newtheorem{theorem}{Theorem}[section]

\newtheorem{corollary}[theorem]{Corollary}

\newtheorem{lemma}[theorem]{Lemma}

\newtheorem*{theorem*}{Theorem}

\theoremstyle{definition}
\newtheorem{definition}[theorem]{Definition}
\newtheorem{remark}[theorem]{Remark}

\newtheorem{observation}{Observation}
\newtheorem{example}[theorem]{Example}

\DeclareMathOperator{\Des}{Des}
\DeclareMathOperator{\des}{des}
\DeclareMathOperator{\negative}{neg}
\DeclareMathOperator{\maj}{maj}
\DeclareMathOperator{\col}{col}
\DeclareMathOperator{\fmaj}{fmaj}
\DeclareMathOperator{\sgn}{sgn}
\DeclareMathOperator{\Negative}{Neg}
\DeclareMathOperator{\C}{\mathfrak{C}}

\makeatletter
\newcommand{\addresseshere}{%
  \enddoc@text\let\enddoc@text\relax
}
\makeatother

\title[Descent sets of cyclic permutations in types B and D]{Descent sets of cyclic permutations \\ in types B and D}
\author{Kevin Liu}
\address{Department of Mathematics and Computer Science, The University Of The South}
\email{keliu@sewanee.edu}
\date{}


\begin{document}

\begin{abstract}
   Elizalde constructed a bijection $\phi$ from the cyclic permutations $\pi\in S_{n+1}$ to the symmetric group $S_n$ satisfying $\Des(\pi)\cap \{1,2,\ldots,n-1\}=\Des(\phi(\pi))$. We give a corresponding result on the signed symmetric group $B_n$ by constructing a function $\Phi$ from the cyclic signed permutations $\pi\in B_{n+1}$ to $B_n$ satisfying $\Des(\pi)\cap \{0,1,\ldots,n-1\}=\Des(\Phi(\pi))$. Moreover, letting $D_{n+1}\subseteq B_{n+1}$ be the subgroup consisting of signed permutations with an even number of sign changes, we show that the restriction of $\Phi$ to the cyclic signed permutations in $D_{n+1}$ or its complement is a bijection. Our function $\Phi$ reduces to Elizalde's original bijection $\phi$ under the natural identification of the symmetric groups as subgroups of the signed symmetric groups. One application of our results is asymptotic normality of the descent and flag major index statistics on the cyclic signed permutations in $B_{n}$ and $D_n$. 
\end{abstract}

\maketitle 

\section{Introduction}

For a permutation $\pi\in S_n$, a descent is any index $i\in \{1,2,\ldots,n-1\}$ satisfying $\pi(i)>\pi(i+1)$, and the descent statistic counts the number of descents in a permutation. These are widely studied properties of permutations and appear in numerous contexts. Examples include card shuffling \cite{BayerDiaconis}, carries when adding numbers \cite{carries_shuffling,holte}, and flag varieties \cite{FULMAN1999390}. See also \cite{Petersen2015} for an extensive treatment of the descent statistic. 

Properties of descents have also been studied on permutations with a fixed cycle type, and the cyclic permutations are a special case. These are the permutations in $S_n$ whose cycle notation consists of a single cycle of length $n$, and we denote the set of cyclic permutations in $S_n$ using $\C_{S,n}$. For cyclic permutations, Elizalde established the following result.

\begin{theorem}\cite[Theorem 1]{elizalde}\label{thm:elizalde}
    For every positive integer $n$, there is a bijection $\phi:\C_{S,n+1}\to S_n$ satisfying
    \[\Des(\pi)\cap \{1,2,\ldots,n-1\}=\Des(\phi(\pi))\]
    for every $\pi\in \C_{S,n+1}$.
\end{theorem}

Elizalde's proof involves an explicit construction of $\phi$. This begins with a Foata-type bijection to map a cyclic permutation $\pi\in \C_{S,n+1}$ to a permutation in $S_n$ in cycle notation, followed by a methodical procedure that swaps specific entries in the cycle notation. These swaps correct discrepancies in the descent sets so that the resulting permutation satisfies $\Des(\pi)\cap \{1,2,\ldots,n-1\}=\Des(\phi(\pi))$. Miraculously, the resulting function $\phi$ is a bijection, which Elizalde showed by constructing its inverse. 

\subsection*{Main Results}

The symmetric groups are the type $A$ Coxeter groups, and our work concerns the types $B$ and $D$ Coxeter groups. The signed symmetric group $B_n$ consists of permutations on $\{\pm1,\pm2,\ldots,\pm n\}$ that satisfy $\pi(-i)=-\pi(i)$ for all $i\in \{1,2,\ldots,n\}$. The permutations where $\{\pi(1),\pi(2),\ldots,\pi(n)\}$ consists of an even number of negatives form a subgroup, $D_n$. Signed permutations have a notion of descents that is largely analogous with the one on $S_n$, except that a descent at position $0$ is possible and occurs when $\pi(1)<0$. This convention is related to its structure as a Coxeter group. See \cite{BB} for details. 

As in $S_n$, elements in $B_n$ can be expressed in cycle notation, which consists of cycles of the form
\[\begin{pmatrix} a_{\ell} &  a_1 & \ldots & a_{\ell-1} \\
\epsilon_1 a_1 & \epsilon_2a_2 & \ldots &  \epsilon_{\ell} a_{\ell}\end{pmatrix},\]
where $a_1,a_2,\ldots,a_{\ell}\in [n]$ and $\epsilon_1,\epsilon_2,\ldots,\epsilon_{\ell}\in \{-1,1\}$. The {cyclic signed permutations} are the elements $\pi\in B_n$ whose cycle notation consists of a single cycle of length $n$. We will use $\C_{B,n}$ to denote the cyclic signed permutations in $B_n$, and we partition $\C_{B,n}$ into $\C_{D,n}$ and $\overline{\C_{D,n}}$ based on whether a cyclic signed permutation is in $D_n$ or its complement. %These sets $\C_{D,n}$ and $\overline{\C_{D,n}}$ are conjugacy classes of $B_n$, similar to how the cyclic permutations in $S_n$ form a conjugacy class. 
Our main result is the following analog of \cref{thm:elizalde} for cyclic signed permutations.

\begin{theorem}\label{thm:main_thm}
    For every positive integer $n$, there exists a function $\Phi:\C_{B,n+1}\to B_n$ with the following properties:
    \begin{enumerate}[label=(\alph*)]
        \item for all $\pi\in \C_{B,n+1}$, we have $\Des(\pi)\cap \{0,1,\ldots,n-1\}=\Des(\Phi(\pi))$, and
        \item the restriction of $\Phi$ to $\C_{D,n+1}$ or $\overline{\C_{D,n+1}}$ is a bijection.
    \end{enumerate}
\end{theorem}

\begin{corollary}\label{thm:main_cor}
    For any positive integer $n$ and $I\subseteq \{0,1,\ldots,n-1\}$, the following are equal:
    \begin{enumerate}[label=(\alph*)]
        \item the number of $\pi\in B_n$ with $\Des(\pi)=I$,
        \item the number of $\pi\in \C_{D,n+1}$ with $\Des(\pi)\cap \{0,1,\ldots,n-1\}=I$, and 
        \item the number of $\pi\in \overline{\C_{D,n+1}}$ with $\Des(\pi)\cap \{0,1,\ldots,n-1\}=I$. 
    \end{enumerate}
\end{corollary}

Our general approach is similar in spirit to Elizalde's original bijection. However, numerous changes are needed to account for the negatives that appear in the cycle notation, and we construct separate inverses for the sets $\C_{D,n+1}$ and $\overline{\C_{D,n+1}}$. An implementation of the functions in this paper using \texttt{Python} is available at \[\text{\url{https://github.com/kliu-math/desBn.git}}.\]
This implementation includes verifying \cref{thm:main_thm} for $\Phi:\C_{B,n+1}\to B_n$ when $n\leq 9$ and on randomly generated signed cyclic permutations for larger values of $n$. In the case where $\pi\in \C_{B,n+1}$ does not have any negatives in its cycle notation, we can view $\pi$ as a permutation in $S_{n+1}$. In this case, our function $\Phi$ reduces to Elizalde's original bijection $\phi$, so \cref{thm:main_thm} implies \cref{thm:elizalde}. Some of our intermediate results for establishing \cref{thm:main_thm} reduce to ones established by Elizalde for proving \cref{thm:elizalde}, while others are specific to signed permutations. 

As one application of our result, we analyze the asymptotic distributions of the descent and flag major index statistics on $\C_{B,n},$ $\C_{D,n}$, and $\overline{\C_{D,n}}$. Our approach is to use \cref{thm:main_thm,thm:main_cor} to approximate the distributions of these statistics on $\C_{B,n},$ $\C_{D,n}$, and $\overline{\C_{D,n}}$ with the corresponding distributions on $B_{n-1}$, which are known to be asymptotically normal \cite{CM2012}. This approximation introduces some error due to the potential descent at position $n$, and we show that this error disappears in the asymptotic distribution to establish the following result. 

\begin{theorem}\label{thm:CLT}
    Let $(X_n)_{n\geq 1}$ be the random variables corresponding to either the descent or flag major index statistic on $(\C_{B,n})_{n\geq 1}$, $(\C_{D,n})_{n\geq 1}$, or $(\overline{\C_{D,n}})_{n\geq 1}$ with means $(\mu_n)_{n\geq 1}$ and variances $(\sigma_n^2)_{n\geq 1}$. The standardized random variable $(X_n-\mu_n)/\sigma_n$ converges in distribution to the standard normal distribution. 
\end{theorem}

\subsection*{Related Work}

Elizalde's bijection in \cref{thm:elizalde} also preserves the set of weak excedances. \cref{thm:elizalde} also inspired subsequent work by Baril \cite{baril}, who constructed another bijection from $\C_{S,n+1}$ to $S_n$. Baril's bijection preserves weak excedances, translates quasi-fixed points into fixed points, and preserves left-to-right maxima. Aside from our \cref{thm:main_thm}, we are not aware of any other bijections from $\C_{D,n+1}$ or $\overline{\C_{D,n+1}}$ to $B_n$ with similar statistic-preserving properties.

There is extensive literature on descents and general cycle type in $S_n$ and $B_n$. For $S_n$, the distribution of the descent statistic by cycle type was derived by Diaconis, McGrath, and Pitman \cite{diaconismcgrathpitman} and also by Fulman \cite{fulman}. The asymptotic normality of the descent statistic by cycle type was established by Kim and Lee \cite{kimlee}. In $B_n$, the distribution of the descent statistic by cycle type is given in \cite[Proposition 5.7]{CLLSY}, which built on prior work of Reiner \cite{reiner} involving a different definition of descents in $B_n$. Another closely related problem is to count the number of permutations or signed permutations with a fixed descent set and cycle type. This was studied in $S_n$ by Gessel and Reutenauer \cite{gessel} and more recently by Elizalde and Troyka \cite{troyka}. In $B_n$, this was studied by Poirier \cite{poirier}, who also derived some general results involving the colored permutation groups $S_{n,r}=\mathbb{Z}_r\wr S_n$. These contain $S_n\cong S_{n,1}$ and $B_n\cong S_{n,2}$ as special cases.  

Our \cref{thm:CLT} for the descent statistic is closely related to \cite[Theorem 1.3]{CLLSY}, which gives asymptotic normality for the descent statistic on any sequence of cycle types in $B_n$ where the limiting number of cycles of any fixed length approaches 0 as $n\to\infty$. This contains $(\C_{D,n})_{n\geq 1}$ and $(\overline{\C_{D,n}})_{n\geq 1}$ as special cases, and one can then derive the corresponding result for $(\C_{B,n})_{n\geq 1}$. While \cref{thm:CLT} for the descent statistic appears to be a special case of \cite[Theorem 1.3]{CLLSY}, the Method of Moments \cite[Section 30]{billingsley} and \cite[Theorem 1.1]{CLLSY} show that these two results for the descent statistic are actually equivalent. The same reasoning shows that \cref{thm:CLT} for the flag major index also translates to any sequence of cycle types where the limiting number of cycles of any fixed length approaches 0 as $n\to\infty$. 

Finally, an analog of \cite[Theorem 1.3]{CLLSY} for the descent and flag major index statistics on the colored permutation groups $S_{n,r}=\mathbb{Z}_r\wr S_n$ is given in \cite[Theorem 1.1]{des_fmaj}. However, the descent and flag major index statistics in $B_n$ do not align with the corresponding statistics on $S_{n,2}$, so these two results are not directly related. However, one can generalize \cref{thm:elizalde} to colored permutation groups and establish an analog of \cref{thm:CLT} in this setting, and we will outline this in an Appendix. By the Method of Moments \cite[Section 30]{billingsley} and \cite[Theorem 1.1]{CLLSY}, this analog of \cref{thm:CLT} is equivalent to \cite[Theorem 1.1]{des_fmaj}.

\subsection*{Outline of Paper}

We begin in \cref{sec:background} with background on signed permutations. In \cref{sec:bijection}, we will construct our function $\Phi:\C_{B,n+1}\to B_n$ and show that it preserves descents in $\{0,1,\ldots,n-1\}$. We then construct its inverses on $\C_{D,n+1}$ and $\overline{\C_{D,n+1}}$ in \cref{sec:inverses} to establish \cref{thm:main_thm} and \cref{thm:main_cor}. We apply these results in \cref{sec:normal} to prove \cref{thm:CLT}. We show that $\Phi$ reduces to Elizalde's original bijection in \cref{appendix:A}, and we also discuss analogs of our results for colored permutation groups in \cref{appendix:B}.

\section{Background and definitions}\label{sec:background}

We begin with preliminary information on signed permutations, including cycle notation and the relevant signed permutation statistics. We refer the reader to \cite{BB} for a general treatment of Coxeter groups, as well as \cite{adin_roichman}, \cite{brenti94},  and \cite{reiner} for additional background specific to signed permutations.

\subsection{Signed Permutations}\label{sec:background1}

For any positive integer $n$, we will use the notation $[n]=\{1,2,\ldots,n\}$ and $[\pm n]=\{\pm1 , \pm 2,\ldots,\pm n\}$. A \emph{signed permutation} is a bijection $\sigma:[\pm n]\to [\pm n]$ satisfying $\sigma(-i)=-\sigma(i)$ for all $i\in [n]$. The property $\sigma(-i)=-\sigma(i)$ shows that $\sigma$ is uniquely determined by the images of elements in $[n]$. The 
\emph{one-line notation} of $\sigma$ is given by writing these elements in the form $[\sigma(1),\sigma(2),\ldots,\sigma(n)]$. 

\begin{example}\label{ex:1}
    Consider the signed permutation $\sigma$ on $[\pm 6]$ defined by $\sigma(1)=-3$, $\sigma(2)=1$, $\sigma(3)=2$, $\sigma(4)=-5$, $\sigma(5)=-4$, and $\sigma(6)=6$. 
    Its one-line notation is $\sigma=[-3,1,2,-5,-4,6].$ From the property $\sigma(-i)=-\sigma(i)$ for $i\in [n]$, we find that $\sigma(-1)=3$, $\sigma(-2)=-1$, $\sigma(-3)=-2$, $\sigma(-4)=5$, $\sigma(-5)=4$, and $\sigma(-6)=-6$.
\end{example}

The \emph{type B Coxeter group} $B_n$, also called the \emph{signed symmetric group}, consists of signed permutations on $[\pm n]$ with group operation given by function composition. For brevity, we will abuse terminology and refer to this group operation as multiplication, and we will indicate it by writing elements of $B_n$ adjacent to one another. As this operation is function composition, it is performed from right to left, so 
\[\pi\sigma=[\pi\circ \sigma(1),\pi\circ \sigma(2), \ldots ,\pi\circ \sigma(n)]\]
for all $\sigma,\pi\in B_n$. The \emph{type $D$ Coxeter group} $D_n$ is the subgroup of $B_n$ consisting of the signed permutations whose one-line notation contains an even number of negatives.

\begin{example}\label{ex:2}
    Let $\pi=[1,-3,-2,5,6,4]\in B_6$. Letting $\sigma\in B_6$ be the permutation from \cref{ex:1}, we have that
    $\pi \sigma = [2,1,-3,-6,-5,4].$ Observe that $\pi\in D_6$, while $\sigma,\pi\sigma\notin D_6$. 
\end{example}

Like in the symmetric group $S_n$, elements in $B_n$ can be expressed in cycle notation. The \emph{two-line cycle notation} of $\sigma\in B_n$ expresses the signed permutation as a product of disjoint cycles of the form
\[\begin{pmatrix} a_{\ell} &  a_1 & \ldots & a_{\ell-1} \\
\epsilon_1 a_1 & \epsilon_2a_2 & \ldots &  \epsilon_{\ell} a_{\ell}\end{pmatrix},\]
where $a_1,a_2,\ldots,a_{\ell}\in [n]$ and $\epsilon_1,\epsilon_2,\ldots,\epsilon_{\ell}\in \{-1,1\}$ are chosen according to $\sigma(a_i)=\epsilon_{i+1}a_{i+1}$ for all $i\in [\ell]$, with the convention that $\epsilon_{\ell+1}a_{\ell+1}=\epsilon_1a_1$. The \emph{(one-line) cycle notation} is obtained by deleting the first line, resulting in
\[(\epsilon_1a_1,\epsilon_2a_2,\ldots,\epsilon_{\ell}a_\ell).\] 
We include commas between elements in the cycle notation for clarity. For brevity, we will occasionally omit cycles of the form $(i)$ for $i\in [n]$, similar to the convention in $S_n$.

\begin{example}\label{ex:3}
    The permutation $\sigma$ from \cref{ex:1} can be expressed as \[\begin{pmatrix} 1 & 3 & 2 \\ -3 & 2 & 1 \end{pmatrix} \begin{pmatrix}
    4 & 5 \\ -5 & -4 
\end{pmatrix} \begin{pmatrix}
    6 \\ -6
\end{pmatrix}=(-3, 2, 1)(-5, -4)(6).\]
Observe that the set of elements that appear in the cycle notation and one-line notation coincide, and we will use this fact freely throughout our work. 
\end{example}

Given a cycle notation $\sigma_1\sigma_2\dots \sigma_m$ for $\sigma\in B_n$, we abuse notation and use $i\in \sigma_j$ to indicate that $i$ appears in the cycle $\sigma_j$. We will also abuse terminology and refer to $\sigma(|i|)$ (resp. $\pm \sigma^{-1}(i)$) as the element following (resp. preceding) $\sigma$ in the cycle notation. 

As with permutations in $S_n$, the cycle notation of a signed permutation in $B_n$ is not unique. For our work, we will often use a specific convention based on one in $S_n$. The cycle notation $\sigma=\sigma_1\sigma_2\dots\sigma_m$ for $\sigma\in B_n$ is \emph{canonical} if
\begin{itemize}
    \item the largest element in each $\sigma_i$ is in the first position, and 
    \item the first elements in $\sigma_1,\sigma_2,\ldots,\sigma_m$ are in increasing order when read from left to right. 
\end{itemize}

\begin{example}
Consider $\sigma=(-3,2,1)(-5,-4)(6)$ from \cref{ex:3}. Rewriting each cycle so that the largest element is first, we obtain $\sigma=(2,1,-3)(-4,-5)(6)$. Rearranging cycles so that first elements are in increasing order, we obtain the canonical cycle notation $\sigma=(-4,-5)(2,1,-3)(6)$. 
\end{example}

Now suppose $\sigma\in B_n$ has been expressed in cycle notation $\sigma_1\sigma_2\dots \sigma_m$. The \emph{length} of a cycle $\sigma_i$ is the number of elements in it. A signed permutation $\sigma\in B_n$ is \emph{cyclic} if its cycle notation consists of a single cycle of length $n$. We will use $\C_{B,n}\subseteq B_n$ to denote the set of cyclic signed permutations in $B_n$. Furthermore, we define two subsets $\C_{D,n}=\C_{B,n}\cap D_n$ and $\overline{\C_{D,n}}=\C_{B,n}\setminus D_n$ based on whether or not elements are in $D_n$.

\subsection{Statistics}\label{sec:background2}

We next describe the signed permutation statistics and probability theory that will be relevant for our work. Throughout, we will assume some general familiarity with random variables.  We will use $P(\,\cdot\,)$ for the probability measure associated to a random variable.

A \emph{signed permutation statistic} is any function from $B_n$ to $\mathbb{R}$, and by equipping $B_n$ with the uniform measure, we will consider a statistic as a random variable. For the statistics in this paper, we first define the \emph{descent set} of $\sigma\in B_n$ by
\[\Des(\sigma)=\{i:0\leq i\leq n-1 \text{ and }\sigma(i)>\sigma(i+1)\},\]
where $0$ is a fixed point by convention. Alternatively, $0\in \Des(\sigma)$ whenever $\sigma(1)<0$. Additionally, the \emph{negative set} of $\sigma\in B_n$ is
\[\Negative(\sigma)=\{i:1\leq i\leq n \text{ and }\sigma(i)<0\}.\] 
Using these, we define the following statistics:
\begin{itemize}
    \item the \emph{descent statistic} is $\des(\sigma)=|\Des(\sigma)|$,
    \item the \emph{major index statistic} is $\maj(\sigma)=\sum_{i\in \Des(\sigma)} \, i $,
    \item the \emph{negative statistic} is $\negative(\sigma)=|\Negative(\sigma)|$, and
    \item the \emph{flag major index statistic} is $\fmaj(\sigma)=2\cdot \maj(\sigma)+\negative(\sigma)$. 
\end{itemize}
Note that the descent and flag major index statistics on $B_n$ are analogs of the descent and major index statistics in $S_n$. Similar to how the major index statistic is equidistributed with the length (or inversion) statistic in $S_n$, the flag major index statistic is equidistributed with the length statistic in $B_n$ \cite{adin_roichman}. 

\begin{example}
    For the permutation $\sigma=[-3,1,2,-5,-4,6]$ from \cref{ex:1}, we have that $\Des(\sigma)=\{0,3\}$ and $\Negative(\sigma)=\{1,3,4\}$. Using these, we calculate $\des(\sigma)=2$, $\maj(\sigma)=3$, $\negative(\sigma)=3,$ and $\fmaj(\sigma)=9$. 
\end{example}

We now recall the definitions for two types of convergence for random variables. We also state a tool involving these two types of convergence. 

\begin{definition}
    Let $(X_n)_{n\geq 1}$ be a sequence of real-valued random variables, and let $X$ be a real-valued random variable. Define $F_n(x)=P(X_n\leq x)$ and $F(x)=P(X\leq x)$ to be their respective cumulative distribution functions. 
    \begin{enumerate}[label=(\alph*)]
        \item $X_n$ \emph{converges in distribution} to $X$, denoted $X_n\xrightarrow{d} X$, if $F_n(x)\to F(x)$ for all $x$ that are continuity points of $F$. 
        \item $X_n$ \emph{converges in probability} to $X$, denoted $X_n\xrightarrow{p} X$, if for all $\epsilon>0$, we have that 
        \begin{equation}\label{eq:converge}
            \lim_{n\to\infty} P(|X_n-X|>\epsilon)= 0.
        \end{equation}
    \end{enumerate}
    It is well known that convergence in probability implies convergence in distribution. 
\end{definition}

\begin{theorem}[Slutsky's Theorem]\label{thm:Slutsky}
    Let $(X_n)_{n\geq 1}$ and $(Y_n)_{n\geq 1}$ be sequences of random variables such that $X_n\xrightarrow{d} X$ and $Y_n\xrightarrow{p} c$, where $c\in \mathbb{R}$ is a constant. Then $X_n+Y_n\xrightarrow{d} X+c$ and $X_nY_n\xrightarrow{d} X\cdot c$.
\end{theorem}

The means, variances, and asymptotic distributions for the descent and flag major index statistics on $B_n$ are known. Throughout, we use $\mathcal{N}(0,1)$ to denote the standard normal distribution. 

\begin{theorem}\cite[Theorems 3.1 and 3.4]{CM2012}\label{thm:CM}
    Let $X_n$ be the random variable corresponding to the descent statistic on elements of $B_n$. Then $X_n$ has mean $\mu_n=n/2$ and variance $\sigma_n^2=(n+1)/12$. Furthermore, as $n\to\infty$, the standardized random variable $(X_n-\mu_n)/\sigma_n$ converges in distribution to $\mathcal{N}(0,1)$.
\end{theorem}

\begin{theorem}\cite[Theorems 4.1 and 4.3]{CM2012}\label{thm:CM2}
     Let $X_n$ be the random variable corresponding to the flag major index statistic on elements of $B_n$. Then $X_n$ has mean $\mu_n=n^2/4$ and variance $\sigma_n^2=(4n^3+6n^2-n)/36$. Furthermore, as $n\to\infty$, the standardized random variable $(X_n-\mu_n)/\sigma_n$ converges in distribution to $\mathcal{N}(0,1)$.
\end{theorem}

\begin{remark}
    The original statements of \cref{thm:CM,thm:CM2} given in \cite{CM2012} involve the more general setting of colored permutation groups. These are wreath products of the form $S_{n,r}=\mathbb{Z}_r\wr S_n$, where $\mathbb{Z}_r$ is the cyclic group of order $r$. They contain $B_n\cong S_{n,2}$ as a special case. See \cref{appendix:B} for further details.
    
    While the descent statistic on $S_{n,2}$ does not align with our descent statistic on $B_n$ under the usual isomorphism between these groups, results of Steingr\'{i}msson \cite{steingrim94} and Brenti \cite{brenti94} show that the distributions of these statistics are the same on $S_{n,2}\cong B_n$. One can also show that the distributions of the two different flag major index statistics align, e.g., see \cite[Remark 4.8]{des_fmaj}. Hence, the current statements in \cref{thm:CM,thm:CM2} involving these statistics on $B_n$ is appropriate. 
\end{remark}

Our work will also consider the descent and flag major index statistics as random variables on $\C_{B,n}$, $\C_{D,n}$, or $\overline{\C_{D,n}}$ by equipping these sets with the uniform measure. For the descent statistic, the following lemma follows from \cite[Theorem 1.1 and Example 3.5]{CLLSY} with standard tools in probability theory.

\begin{lemma}\label{thm:CLLSY}
    Let $X_n$ be the random variable corresponding to the descent statistic on $\C_{B,n},\C_{D,n}$, or $\overline{\C_{D,n}}$, and let $Y_n$ be the corresponding random variable on 
    $B_n$. If $n\geq 5$, then the mean and variance of $X_n$ are the same as the mean and variance of $Y_n$, respectively. 
\end{lemma}

One can also apply the methods in \cite{CLLSY} to the flag major index. It straightforward to show an analog of \cref{thm:CLLSY} for this statistic. 
 
\begin{lemma}\label{thm:CLLSY2}
    Let $X_n$ be the random variable corresponding to the flag major index statistic on $\C_{B,n},\C_{D,n}$, or $\overline{\C_{D,n}}$, and let $Y_n$ be the corresponding random variable on 
    $B_n$. If $n\geq 5$, then the mean and variance of $X_n$ are the same as the mean and variance of $Y_n$, respectively. 
\end{lemma}

\section{Mapping cyclic signed permutations to signed permutations}\label{sec:bijection}

In this section, we will define our function $\Phi:\C_{B,n+1}\to B_n$ and show that it satisfies $\Des(\pi)\cap \{0,1,2,\ldots,n-1\}=\Des(\Phi(\pi))$ for all $\pi\in \C_{B,n+1}$. Throughout, we will use $\C_{B,n+1}^+$ (resp. $\C_{B,n+1}^-$) for the subset of $\C_{B,n+1}$ consisting of permutations where $n+1$ (resp. $-(n+1)$) appears in the cycle or one-line notation. Our approach begins with an adaptation of Elizalde's algorithm in \cite{elizalde} to elements in $\C_{B,n+1}^+$. After establishing several properties of the algorithm, we construct $\Phi$ by applying variations of this algorithm to elements in $\C_{B,n+1}$.

\subsection{An algorithm on cyclic signed permutations}\label{sec:algorithm_on_cycles}

Throughout, fix a positive integer $n$. We first construct an algorithm that starts with $\pi\in \C_{B,n+1}^+$ and outputs a permutation $\sigma\in B_n$. We will need the following definitions.

\begin{definition}
    Let $\pi\in \C_{B,n+1}$ and $\sigma\in B_n$. For brevity, we will use the notation
    \[\Des_{\Delta}(\pi,\sigma)=\Des(\pi)\Delta \Des(\sigma),\]
    where $\Delta$ denotes the symmetric difference of two sets. Using this, define for any nonnegative integers $x$ and $y$ the Boolean function
    \[P_{\pi,\sigma}(x,y)=\begin{cases}
        \texttt{True} & \text{if $|x-y|=1$ and $\min\{x,y\}\in \Des_{\Delta}(\pi,\sigma)\cap [n-1]$} \\
        \texttt{False} & \text{otherwise.}
    \end{cases}\]
\end{definition}

\begin{definition}
    In a cycle notation for $\pi\in \C_{B,n+1}$, an element $x\in [\pm n]$ is a \emph{left-to-right maximum} if it is larger than every element that appears to its left.
\end{definition}

Additionally, we use $\sgn:\mathbb{Z}\to \{-1,0,1\}$ to denote the sign function
\[\sgn(x)=\begin{cases} 1 & \text{ if $x>0$} \\
0 & \text{ if $x=0$} \\
-1 & \text{ if $x<0$.} \end{cases}\]
Combining these definitions, we now define an algorithm and function $\phi:\C_{B,n+1}^+\to B_n$. For cyclic permutations in $S_{n+1}$ under its natural identification as a subgroup of $B_{n+1}$, the definition of $\phi$ coincides with Elizalde's original bijection from \cref{thm:elizalde}. This reduction is not immediate, so we postpone a detailed description until \cref{appendix:A}.

\begin{definition}\label{def:phi}
    Define $\phi: \C_{B,n+1}^+\to B_n$ to be the function mapping each $\pi\in \C_{B,n+1}^+$ to its output from Algorithm \ref{algorithm}.
\end{definition}

\begin{algorithm}\DontPrintSemicolon
\caption{\algname{Cyclic to Signed Permutation}}\label{algorithm}
\KwIn{$\pi\in \C_{B,n+1}^+$}
\KwOut{a permutation in $B_n$}
express $\pi$ in the cycle notation with $n+1$ in the final position \\
$\pi_{i_1},\pi_{i_2},\ldots,\pi_{i_{m+1}}\coloneqq $ the left-to-right maxima in the cycle notation above\\
set $\sigma=\sigma_1\sigma_2\dots \sigma_{m}\coloneqq (\pi_{i_1},\ldots,\pi_{i_2-1})(\pi_{i_2},\ldots,\pi_{i_3-1})\dots (\pi_{i_{m}},\ldots,\pi_{i_{m+1}-1})$ \label{line:starting}\\
\For{$j=1,2,\ldots,m$\label{line:forloop}} 
{
    $z\coloneqq $ the rightmost entry of $\sigma_j$ \\
    \If{\normalfont $P_{\pi,\sigma}(|z|,|z-1|)$ or $P_{\pi,\sigma}(|z|,|z+1|)$\label{line:if0}}
    {
        set $\epsilon\in \{-1,1\}$ to be the value such that $P_{\pi,\sigma}(|z|,|z+\epsilon|)$ is \texttt{True} and $\pi(|z+\epsilon|)$ is largest\\
        \While{\normalfont $P_{\pi,\sigma}(|z|,|z+\epsilon|)$\label{line:whileloop1}} 
        {
            $x\coloneqq z$ \\
            $y\coloneqq $ whichever of $z+\epsilon$ or $-(z+\epsilon)$ that appears in $\sigma$\\
            \While{\normalfont $P_{\pi,\sigma}(|x|,|y|)$\label{line:whileloop2}}
            {
                respectively replace $x$ and $y$ with $\sgn(x)\cdot |y|$ and $\sgn(y)\cdot |x|$ in the cycle notation of $\sigma$ \\
                \If{\normalfont the replacement did not involve the first element of $\sigma_j$\label{line:if1}}
                {
                $x,y\coloneqq $ the elements respectively preceding the ones replaced
                }
            }
            $z\coloneqq $ the rightmost entry of $\sigma_j$
        }
    }
}
\Return $\sigma$
\end{algorithm}

Throughout, we will refer to the step where $x$ and $y$ are replaced with $\sgn(x)\cdot |y|$ and $\sgn(y)\cdot |x|$ as a \emph{swap}. Additionally, $y$ is chosen based on which of $z+\epsilon$ or $-(z+\epsilon)$ appears in the cycle notation, but for brevity, we will abuse notation and refer to this element as $\pm (z+\epsilon)$. We now give examples to illustrate the algorithm.

\begin{example}\label{ex:phi}
    Consider the signed permutation \[\pi = (-4, -1, 2, 5, -3, -6, 7)=[2,5,-6,-1,-3,7,-4] \in \C_{B,7}^+.\] Its descent set is $\{2,4,6\}$. After finding the left-to-right maxima of the cycle notation, Algorithm \ref{algorithm} begins with
    \[\sigma= (-4)(-1)(2)(5,-3,-6)=[-1,2,-6,-4,-3,5].\]
     In the first iteration of the \texttt{for} loop in line \ref{line:forloop}, we see that $P_{\pi,\sigma}(4,5)$ is \texttt{True} and $P_{\pi,\sigma}(4,3)$ is \texttt{False}. Then the algorithm chooses $\epsilon=1$, and the first replacement in the \texttt{while} loop on line \ref{line:whileloop2} results in
    \[\sigma=(-5)(-1)(2)(4,-3,-6)=[-1,2,-6,-3,-5,4].\]
    Since $P_{\pi,\sigma}(5,6)$ is \texttt{False}, the first iteration of the \texttt{for} loop ends here. In the second iteration of the \texttt{for} loop, we see that $P_{\pi,\sigma}(1,2)$ is \texttt{False} and $P_{\pi,\sigma}(1,0)$ is always \texttt{False}, so no swaps occur. In the third iteration, again no swaps occur, as $P_{\pi,\sigma}(2,1)$ and $P_{\pi,\sigma}(2,3)$ are both \texttt{False}. Consequently, \[\phi(\pi)=(-5)(-1)(2)(4,-3,-6)=[-1,2,-6,-3,-5,4].\] Observe that $\Des(\phi(\pi))=\{0,2,4\}$, so its descents in $[5]$ coincide with those of $\pi$. However, the descent $0\in \Des(\phi(\pi))$ is introduced by the algorithm, and the descent $6\in \Des(\pi)$ cannot occur for $\phi(\pi)\in B_6$.  
\end{example}

\begin{example}\label{ex:phi2}
    For a larger example, consider 
    \begin{equation*}
        \begin{split}
            \pi & =(1,-4,8,-6, 11, 2, -3, 7,  -5, 10, 12,  9,13) \\
            & =[-4,-3,7,8,10,11,-5,-6,13,12,2,9,1]\in \C_{B,13}^+,
        \end{split}
    \end{equation*}
    which has descent set $\{0,6,7,9,10,12\}$. The algorithm begins with
    \begin{equation*}
        \begin{split}
            \sigma & =(1,-4)(8,-6)(11, 2, -3, 7,-5, 10)(12, 9) \\ & =[-4,-3,7,1,10,8,-5,-6,12,11,2,9].
        \end{split}
    \end{equation*}

    In the first iteration of the \texttt{for} loop, we see that $P_{\pi,\sigma}(4,3)$ is \texttt{True} and $P_{\pi,\sigma}(4,5)$ is \texttt{False}, so the algorithm performs swaps beginning with $4$ and $3$ to obtain
    \begin{equation*}
        \begin{split}
            \sigma & =(2,-3)(8,-6)(11, 1, -4, 7, -5, 10)(12, 9) \\ & =[-4,-3,2,7,10,8,-5,-6,12,11,1,9].
        \end{split}
    \end{equation*}
    Since $P_{\pi,\sigma}(3,2)$ is \texttt{False}, the first iteration of the \texttt{for} loop ends here. 
    
    In the second iteration, $P_{\pi,\sigma}(6,5)$ is \texttt{True} and $P_{\pi,\sigma}(6,7)$ is \texttt{False}, so the algorithm performs swaps beginning with $6$ and $5$ to obtain
    \begin{equation*}
        \begin{split}
            \sigma & =(2,-3)(7,-5)(11, 1, -4,8,- 6, 10)(12,9) \\ & =[-4,-3,2,8,7,10,-5,-6,12,11,1,9].
        \end{split}
    \end{equation*}
    Since $P_{\pi,\sigma}(5,4)$ is now \texttt{True}, the algorithm performs further swaps beginning with $5$ and $4$ to obtain
    \begin{equation*}
        \begin{split}
            \sigma & =(2,-3)(7,-4)(11, 1, -5,8, -6, 10)(12,9) \\ & =[-5,-3,2,7,8,10,-4,-6,12,11,1,9].
        \end{split}
    \end{equation*}
    As $P_{\pi,\sigma}(4,3)$ is \texttt{False}, the second iteration of the \texttt{for} loop end here.

    In the third iteration, both $P_{\pi,\sigma}(10,11)$ and $P_{\pi,\sigma}(10,9)$ are \texttt{False}. Hence, the algorithm does not perform any swaps, so
    \begin{equation*}
        \begin{split}
            \phi(\pi) & =(2,-3)(7,-4)(11, 1, -5,8, -6, 10)(12,9) \\ & =[-5,-3,2,7,8,10,-4,-6,12,11,1,9].
        \end{split}
    \end{equation*}
    This has descent set $\{0,6,7,9,10\}$, and this matches $\Des(\pi)$ at all elements except $12$, which is not a possible descent for $\phi(\pi)\in B_{12}$.
\end{example}

With $\phi$ defined, we now consider its various properties. We begin by describing properties of the signed permutation $\sigma$ at the start of the algorithm.

\begin{lemma}\label{lem:algstart}
    Consider the application of Algorithm \ref{algorithm} on $\pi\in \C_{B,n+1}^+$, let \[\sigma=\sigma_1\sigma_2\dots \sigma_{m}\coloneqq (\pi_{i_1},\ldots,\pi_{i_2-1})(\pi_{i_2},\ldots,\pi_{i_3-1})\ldots (\pi_{i_{m}},\ldots,\pi_{i_{m+1}-1})\] be the signed permutation formed at the start of the algorithm in line \ref{line:starting}, and define \[L=\{|\pi_{i_2-1}|,|\pi_{i_3-1}|,\ldots, |\pi_{i_{m+1}-1}|\}.\]
    \begin{enumerate}[label=(\alph*)]
        \item We have that $\pi_{i_1}<\pi_{i_2}<\dots < \pi_{i_m}$. Furthermore, $\pi_{i_j}$ is the largest element in $\sigma_1,\sigma_2,\ldots,\sigma_j$ for each $j\in [m]$. 
        \item For any $x\in [n]\setminus L$, we have that $\pi(x)=\sigma(x)$, and for any $x\in L$, we have that $\pi(x)>\sigma(x)$. 
        \item If an element $x\in [n-1]$ is in $\Des_{\Delta}(\pi,\sigma)$, then one of the following must be true:
        \begin{itemize}
            \item $x\in L$, $x+1\notin L$, $\pm (x+1)$ appears right of $\pm x$ in $\sigma$, and 
            $\pi(x)>\pi(x+1)=\sigma(x+1)>\sigma(x),$
            or
            \item $x\notin L$, $x+1\in L$, $\pm x$ appears right of $\pm (x+1)$ in $\sigma$, and
            $\pi(x+1)>\pi(x)=\sigma(x)>\sigma(x+1).$
        \end{itemize}
        In particular, $\pm x$ and $\pm (x+1)$ cannot be in the same cycle.
    \end{enumerate} 
\end{lemma}

\begin{proof}
    For (a), the algorithm creates the individual cycles so that left-to-right maxima are the first elements in their respective cycles. From this, we conclude $\pi_{i_1}<\pi_{i_2}<\dots < \pi_{i_m}$. Additionally, $\pi_{i_j}$ is the right-most left-to-right maximum in $\sigma_1,\sigma_2,\ldots,\sigma_j$, so it must be larger than all elements in those cycles. 

    For (b), the construction of $\sigma$ also implies that the element following any $x\in [n]\setminus L$ in the cycle notations of $\pi$ and $\sigma$ coincide, so $\pi(x)=\sigma(x)$. For each $x\in L$, we find the cycle $\sigma_j$ containing $\pm x$ as the last element and use (a) to conclude
    \[\pi(x)=\pi_{i_{j+1}}>\pi_{i_{j}}=\sigma(x),\]
    so (b) follows. 
    
    For (c), note that if both $x,x+1\notin L$, then (b) implies $\pi(x)=\sigma(x)$ and $\pi(x+1)=\sigma(x+1)$, so $x\notin \Des_{\Delta}(\pi,\sigma)$. If both $x,x+1\in L$, then let $\sigma_j$ and $\sigma_k$ be the cycles containing $\pm x$ and $\pm (x+1)$ in the last entry, respectively. Using (a), the relative order of $\pi(x)=\pi_{i_{j+1}}$ and $\pi(x+1)=\pi_{i_{k+1}}$ matches that of $\sigma(x)=\pi_{i_j}$ and $\sigma(x+1)=\pi_{i_k}$, so $x\notin \Des_{\Delta}(\pi,\sigma)$. Combined, we conclude $x\in \Des_{\Delta}(\pi,\sigma)$ requires that exactly one of $x$ or $x+1$ is in $L$. The two possibilities given in (c) now follow from (a), (b), and casework on which of $x$ or $x+1$ is in $L$.
\end{proof}

\begin{lemma}\label{cor:algstart}
    Consider the application of Algorithm \ref{algorithm} on $\pi\in \C_{B,n+1}^+$, and let \[\sigma=\sigma_1\sigma_2\ldots \sigma_{m}\coloneqq (\pi_{i_1},\ldots,\pi_{i_2-1})(\pi_{i_2},\ldots,\pi_{i_3-1})\dots (\pi_{i_{m}},\ldots,\pi_{i_{m+1}-1})\] be the signed permutation at the start of Algorithm \ref{algorithm}. Suppose $x$ is the last element in some cycle of $\sigma$ and  $\epsilon\in \{-1,1\}$. If $P_{\pi,\sigma}(|x|,|x+\epsilon|)$ is \texttt{True} and $P_{\pi,\sigma}(|x|,|x-\epsilon|)$ is \texttt{False}, then $\sigma(|x|)<\sigma(|x-\epsilon|)<\sigma(|x+\epsilon|)$ cannot occur.
\end{lemma}

\begin{proof}
    Suppose that $x$ is the last element in $\sigma_j$. Since $P_{\pi,\sigma}(|x|,|x+\epsilon|)$ is \texttt{True}, \cref{lem:algstart}(c) implies that
    \begin{equation}\label{eq:algstart}
    \pi_{i_{j+1}}=\pi(|x|)>\pi(|x+\epsilon|)=\sigma(|x+\epsilon|)>\sigma(|x|)=\pi_{i_j}.
    \end{equation}
    There are two cases for $P_{\pi,\sigma}(|x|,|x-\epsilon|)$ being \texttt{False}. If $\sigma(|x|)>\sigma(|x-\epsilon|)$ and $\pi(|x|)>\pi(|x-\epsilon|)$, then combined with \eqref{eq:algstart}, we conclude \[\sigma(|x+\epsilon|)>\sigma(|x|)>\sigma(|x-\epsilon|).\] Alternatively, suppose that $\sigma(|x|)<\sigma(|x-\epsilon|)$ and $\pi(|x|)<\pi(|x-\epsilon|)$. If $\pm (x-\epsilon)$ is not the last element of a cycle, then \cref{lem:algstart}(b) implies $\sigma(|x-\epsilon|)=\pi(|x-\epsilon|)$, so combined with \eqref{eq:algstart}, we conclude \[\sigma(
    |x-\epsilon|)=\pi(|x-\epsilon|)>\pi(|x|)>\sigma(|x+\epsilon|)>\sigma(|x|).\]
    If $\pm (x-\epsilon)$ is the last element of a cycle, then \cref{lem:algstart}(a) with $\pi(|x|)<\pi(|x-\epsilon|)$ implies $\pm (x-\epsilon)$ appears to the right of $\pm x$. Combined with \eqref{eq:algstart}, we find
    \[\sigma(|x-\epsilon|)\geq \pi_{i_{j+1}}>\sigma(|x+\epsilon|) >\sigma(|x|).\]
    In all cases, $\sigma(|x|)<\sigma(|x-\epsilon|)<\sigma(|x+\epsilon|)$ is not possible. 
\end{proof}

\subsection{Swaps and descents throughout the \texttt{for} loop}\label{sec:for_loop}

We next consider what occurs in Algorithm \ref{algorithm} throughout the \texttt{for} loop. We will show that for any $\pi\in \C_{B,n+1}^+$, the outputted permutation $\phi(\pi)$ satisfies $\Des(\pi)\cap [n-1]=\Des(\phi(\pi))$. We then consider possibilities for the descent at position $0$ for $\pi$ and $\phi(\pi)$.

Throughout this section, we will fix some notation. Let $\pi\in \C_{B,n+1}^+$ be a cyclic signed permutation inputted into Algorithm \ref{algorithm}, and let
\begin{equation}
    \pi_1\pi_2\dots \pi_m=(\pi_{1,1},\ldots,\pi_{1,\ell_1})(\pi_{2,1},\ldots,\pi_{2,\ell_2})\dots (\pi_{m,1},\ldots,\pi_{m,\ell_m})
\end{equation}
be the permutation in line \ref{line:starting} of the algorithm, which is formed by creating cycles using the left-to-right maxima of $\pi$. Our focus will be on any signed permutation at the start of an iteration of the \texttt{for} loop in line \ref{line:forloop} or the \texttt{while} loop in line \ref{line:whileloop1}, which we will denote using
\begin{equation}
        \begin{split}
            \sigma & =\sigma_1\sigma_2\dots\sigma_m=(\sigma_{1,1},\ldots,\sigma_{1,\ell_1})(\sigma_{2,1},\ldots,\sigma_{2,\ell_2})\dots(\sigma_{m,1},\ldots,\sigma_{m,\ell_m}).
        \end{split}
    \end{equation}
We will assume that the current iteration of the \texttt{for} loop is the $j$-th iteration, so any swaps that are being performed involve the elements in $\sigma_j$ and another cycle. We will be primarily interested in the following properties of $\sigma$.

\begin{definition} 
The \emph{order properties} of ${\sigma}$ are the following.  
\begin{enumerate}[label=(\Alph*)]
    \item The relative order of elements within each cycle $\sigma_1,\ldots,\sigma_m$ matches the relative order within the respective cycle of $\pi_1,\ldots,\pi_m$. In particular, the largest element in each cycle is the first one. \label{op:relative}
    \item The first elements of the cycles satisfy $\sigma_{1,1}<\sigma_{2,1}<\dots <\sigma_{m,1}$. Additionally, for every $i\in [m]$, the first element $\sigma_{i,1}$ is the largest element that appears in $\sigma_1,\sigma_2,\ldots,\sigma_i$. \label{op:first}
    \item Suppose $x\in [\pm n]$ appears in the cycle notation of $\sigma$, and let $k\geq j+1$. Then $\pi(|x|)>\pi_{k,1}$ if and only if $\sigma(|x|)\geq \sigma_{k,1}$, where equality holds when $x\in \sigma_k$ is the last element. Furthermore, any $x$ satisfying these equivalent properties has not been involved in any swaps.\label{op:large}
    \item If $d\in [n-1]$ is in symmetric difference $\Des_{\Delta}(\pi,\sigma)$, then exactly one of $\pm d$ and $\pm (d+1)$ is the last element in some ${\sigma}_{j},\ldots,{\sigma}_{m}$ and the remaining element appears in the non-last position of some cycle to its right. Furthermore, letting $x$ be the element that appears in a last position and $\pm (x+\epsilon)$ with $\epsilon\in \{-1,1\}$ be the other element,  it must be that $\pi(|x|)>\pi(|x+\epsilon|)$ and ${\sigma}(|x|)<{\sigma}(|x+\epsilon|)$. \label{op:descents}
\end{enumerate}
\end{definition}

Our general approach to showing the above properties hold is to use strong induction. The next lemma is our base case. 

\begin{lemma}\label{lem:basecase}
    The order properties hold when $\sigma$ is the initial permutation in line \ref{line:starting} of Algorithm \ref{algorithm}.
\end{lemma}

\begin{proof}
    The algorithm starts with $\sigma=\sigma_1\sigma_2\dots \sigma_m=\pi_1\pi_2\dots \pi_m$. Each of the claims in the order properties then follow immediately from this or from \cref{lem:algstart}.
\end{proof}

Showing the order properties hold for any $\sigma$ will require a careful analysis of the swaps performed in the \texttt{while} loop in line \ref{line:whileloop1}, which may include multiple iterations of the \texttt{while} loop in line \ref{line:whileloop2}.  We will first consider the case where no swaps are performed.

\begin{lemma}\label{lem:no_swaps}
    Suppose the order properties hold for $\sigma$ and no swaps are performed in this iteration of the \texttt{for} loop. Then the order properties hold for $\sigma$ considered as the signed permutation at the start of iteration $j+1$ of the \texttt{for} loop. 
\end{lemma}

\begin{proof}
    Order properties \ref{op:relative} and \ref{op:first} are independent of the \texttt{for} loop iteration, so these are immediate. By starting with order property \ref{op:large} for $\sigma$ as the signed permutation in iteration $j$ and specializing to $k\geq j+2$, we conclude order property \ref{op:large} for $\sigma$ as the signed permutation in iteration $j+1$. For order property \ref{op:descents}, let $z\in \sigma_j$ be the last element. As the \texttt{for} loop does not perform any swaps, it must be that both $P_{\pi,\sigma}(|z|,|z-1|)$ and $P_{\pi,\sigma}(|z|,|z+1|)$ are \texttt{False}. Hence, $z\in \sigma_j$ is not involved in any descent in $\Des_{\Delta}(\pi,\sigma)\cap [n-1]$, so $\sigma_j$ can be omitted from order property \ref{op:descents} for $\sigma$ as the signed permutation in iteration $j$. This results in order property \ref{op:descents} when $\sigma$ is considered as the signed permutation in iteration $j+1$.
\end{proof}

We now consider the case where swaps are performed. Denote the signed permutation resulting from the swaps in the \texttt{while} loop of line \ref{line:whileloop2} as 
    \begin{equation}
        \begin{split}
            \sigma' & =\sigma_1'\sigma_2'\dots\sigma_m'=(\sigma_{1,1}',\ldots,\sigma_{1,\ell_1}')(\sigma_{2,1}',\ldots,\sigma_{2,\ell_2}')\dots(\sigma_{m,1}',\ldots,\sigma_{m,\ell_m}').
        \end{split}
    \end{equation} 
Our focus will be on the following properties involving the swaps that produce $\sigma'$.

\begin{definition}
The \emph{swap properties} of ${\sigma}$ are the following statements about the swaps performed on $\sigma$ by the \texttt{while} loop in line \ref{line:whileloop2}. 
    \begin{enumerate}[label=(\Roman*)]
        \item All swaps occur between an element in the $j$-th cycle and an element in a cycle to its right. \label{sp:right}
        \item The last elements in cycles $\sigma_{j+1},\ldots,\sigma_m$ are not affected by swaps. Furthermore, if the first swap is between $z\in \sigma_j$ and $\pm (z+\epsilon)$ where $\epsilon\in \{-1,1\}$, then no element right of $\pm (z+\epsilon)$ in the cycle notation of $\sigma$ is affected by a swap. \label{sp:last}
        \item If an element $x$ in the cycle notation of $\sigma$ satisfies $\sigma(|x|)\geq \sigma_{j+1,1}$, then $x$ is not affected by swaps. \label{sp:large}
        \item Suppose the last element $z\in \sigma_j$ is swapped with $\pm (z+\epsilon)$ where $\epsilon\in \{-1,1\}$. Then the following statements hold for the cumulative effects of the swaps on $\Des_{\Delta}(\pi,\sigma)\cap [n-1]$. \label{sp:descents}
        \begin{itemize}
            \item The descent $\min\{|z|,|z+\epsilon|\}\in \Des_{\Delta}(\pi,\sigma)$ is removed from the symmetric difference.
            \item Suppose $\min\{|z|,|z-\epsilon|\}\in [n-1]$. If $\min\{|z|,|z-\epsilon|\}\in \Des_{\Delta}(\pi,\sigma)$, then this is removed from the symmetric difference, and if $\min\{|z|,|z-\epsilon|\}\notin \Des_{\Delta}(\pi,\sigma)$, then this is not introduced. 
            \item Suppose $\min\{|z+\epsilon|,|z+2\epsilon|\}\in [n-1]$. If $\min\{|z+\epsilon|,|z+2\epsilon|\}\in \Des_{\Delta}(\pi,\sigma)$, then this descent is removed. If $\min\{|z+\epsilon|,|z+2\epsilon|\}\notin \Des_{\Delta}(\pi,\sigma)$ and $\sigma(|z+2\epsilon|)>\sigma(|z+\epsilon|)$, then this descent is not introduced to the symmetric difference. 
            \item Any value in $[n-1]$ that is not $\min\{|z+\epsilon|,|z+2\epsilon|\}$ cannot be introduced into the symmetric difference. 
        \end{itemize}
\end{enumerate}
\end{definition}

Our goal is to show that the swap properties hold for $\sigma$, which will then allow us to show that the order properties hold for $\sigma'$. We begin by describing when the \texttt{while} loop in line \ref{line:whileloop2} continues or terminates.

\begin{lemma}\label{lem:continued_swaps}
    Suppose the order properties hold for $\sigma$. Then the \texttt{while} loop in line \ref{line:whileloop2} continues until one of the following occurs:
    \begin{enumerate}[label=(\alph*)]
        \item the first element of $\sigma_j$ is involved in a swap,
        \item the swapped elements $x$ and $y$ do not have the same sign, or
        \item the absolute values of the elements preceding $x$ and $y$ differ by more than $1$. 
    \end{enumerate}
\end{lemma}

\begin{proof}
    Let $z$ be the last element of $\sigma_j$, and let $\pm(z+\epsilon)$ with $\epsilon\in \{-1,1\}$ be the element that it is swapped with. A swap occurs due to $P_{\pi,\sigma}(|x|,|y|)$ being \texttt{True}, where $x=z$ and $y=\pm (z+\epsilon)$. If $x$ is the first element of $\sigma_j$, then by the construction of the \texttt{while} loop, it will terminate after this swap. 
    
    Otherwise, let $\tilde{\sigma}$ be the permutation after $x$ and $y$ are swapped, and consider the elements $x'$ and $y'$ that precede the ones swapped. Order property \ref{op:descents} for $\sigma$ implies that $P_{\pi,\sigma}(|x'|,|y'|)$ is \texttt{False}. If $x$ and $y$ did not have the same sign, then the elements that follow $x'$ and $y'$ in the cycle notations of $\sigma$ and $\tilde{\sigma}$ have the same relative order, so $P_{\pi,\tilde{\sigma}}(|x'|,|y'|)$ is also \texttt{False}. If $|x'|$ and $|y'|$ differ by more than $1$, then $P_{\pi,\tilde{\sigma}}(|x'|,|y'|)$ is \texttt{False} from its definition. Hence, the \texttt{while} loop will terminate in both of these cases at $\tilde{\sigma}$.
    
    In the remaining cases, $x$ and $y$ must be consecutive elements in $[\pm n]$, and $|x'|$ and $|y'|$ must be consecutive elements in $[n]$. Swapping $x$ and $y$ in $\sigma$ causes $P_{\pi,\tilde{\sigma}}(|x'|,|y'|)$ to be \texttt{True}, and the \texttt{while} loop will continue. Iterating this argument, we see that the \texttt{while} loop will continue until one of the three stated situations occurs. 
\end{proof}

We now give a sequence of lemmas that establish individual swap properties for $\sigma$ and order properties for $\sigma'$. As our approach is to use strong induction, we will typically assume that the order properties and/or swap properties hold for signed permutations at the start of previous iterations of the \texttt{for} loop in line \ref{line:forloop} or the \texttt{while} loop in line \ref{line:whileloop1}. We will refer to these signed permutations as ones \emph{prior} to $\sigma$ in Algorithm \ref{algorithm}. Note that these specifically exclude the intermediate permutations produced during the \texttt{while} loop in line \ref{line:whileloop2}. Additionally, individual results typically do not need all of the order and swap properties, but the exact dependence is overly technical, so for simplicity, we will just assume the entirety of these properties.

\begin{lemma}\label{lem:SPOP1}
    Suppose the order properties hold for $\sigma$. Then swap property \ref{sp:right} holds for $\sigma$, and order property \ref{op:relative} holds for $\sigma'$. 
\end{lemma}

\begin{proof}
    Order property \ref{op:descents} for $\sigma$ implies that the first swap in the \texttt{while} loop is between $z$, the last element of $\sigma_j$, and $\pm(z+\epsilon)$ with $\epsilon\in \{-1,1\}$, which must be a non-last element in some $\sigma_k$ with $k>j$. Consequently, swaps throughout the \texttt{while} loop only occur between $\sigma_j$ and $\sigma_{k}$, so swap property \ref{sp:right} holds for $\sigma$. 
    
    Order property \ref{op:relative} for $\sigma$ implies that the relative order of elements in each cycle of $\sigma_1,\sigma_2,\ldots,\sigma_m$ matches the corresponding ones in $\pi_1,\pi_2,\ldots,\pi_m$. Swap property \ref{sp:right} for $\sigma$ implies that swaps only occur between elements in different cycles. Each of these swaps increases or decreases a single entry in each cycle by $1$, so the relative order of elements within each cycle cannot change. We conclude order property \ref{op:relative} holds for $\sigma'$.
\end{proof}

\begin{lemma}\label{lem:plus1}
    Suppose the order properties hold for $\sigma$. If $\sigma_{j+1,1}=\sigma_{j,1}+1$, then swaps do not occur between $\sigma_j$ and $\sigma_{j+1}$.
\end{lemma}

\begin{proof}
    Assume by contradiction that $\sigma_{j+1,1}=\sigma_{j,1}+1$ and a swap occurs between $\sigma_j$ and $\sigma_{j+1}$. The first swap is between the last element $z\in \sigma_j$ and $\pm (z+\epsilon)\in \sigma_{j+1}$ due to $\min\{|z|,|z+\epsilon|\}\in \Des_{\Delta}(\pi,\sigma)$. Order property \ref{op:descents} for $\sigma$ implies $\pm(z+\epsilon)$ is not the last element of $\sigma_{j+1}$ and $\pi(|z|)>\pi(|z+\epsilon|)$. Since $\pm(z+\epsilon)$ is not the last element of $\sigma_{j+1}$, order property \ref{op:first} for $\sigma$ implies that $\sigma(|z+\epsilon|)<\sigma_{j+1,1}=\sigma_{j,1}+1$. Since $\pm(z+\epsilon)\notin \sigma_j$, we have that $\sigma(|z+\epsilon|)\neq \sigma_{j,1}$, so combined with $\sigma(|z+\epsilon|)<\sigma_{j+1,1}=\sigma_{j,1}+1$, we conclude $\sigma(|z+\epsilon|)<\sigma_{j,1} = \sigma(|z|)$. This aligns with the relative order of $\pi(|z+\epsilon|)$ and $\pi(|z|)$, which contradicts $\min\{|z|,|z+\epsilon|\}\in \Des_{\Delta}(\pi,\sigma)$.
\end{proof}

\begin{corollary}\label{cor:plus1}
    Suppose the order properties hold for $\sigma$ and the \texttt{while} loop performs swaps between elements in $\sigma_j$ and $\sigma_{k}$ with $k>j$. If a swap involving $\sigma_{k,1}$ occurs, then $\sigma_{k,1}>0$, and this swap must be with $-\sigma_{k,1}+ \epsilon\in \sigma_j$ for some $\epsilon \in \{-1,1\}$. 
\end{corollary}

\begin{proof}
    If $k>j+1$, then order property \ref{op:first} for $\sigma$ implies $\sigma_{k,1}>\sigma_{j+1,1}\geq \sigma_{j,1}+1$. If $k=j+1$, then \cref{lem:plus1} implies $\sigma_{k,1}\neq \sigma_{j,1}+1$, so combined with order property \ref{op:first} for $\sigma$, we similarly conclude that $\sigma_{k,1}>\sigma_{j,1}+1$. Order property \ref{op:relative} for $\sigma$ implies $\sigma_{j,1}$ is the largest element in $\sigma_j$, so in both cases, we see that $\sigma_{k,1}\pm 1\notin \sigma_j$. As swaps can only involve pairs of elements whose absolute values differ by $1$, we conclude that a swap involving $\sigma_{k,1}\in \sigma_k$ can only occur if this is swapped with $-\sigma_{k,1}+\epsilon\in \sigma_j$ for some $\epsilon\in \{-1,1\}$. Order property \ref{op:first} for $\sigma$ implies that $\sigma_{k,1}$ is larger than all elements in $\sigma_j$, so it must also be that $\sigma_{k,1}>0$ and $-\sigma_{k,1}+\epsilon<0$.
\end{proof}

\begin{lemma}\label{lem:SPOP2}
    Suppose the order properties hold for $\sigma$.  Then swap property \ref{sp:last} holds for $\sigma$, and order property \ref{op:first} holds for $\sigma'$. 
\end{lemma}

\begin{proof}
    By order property \ref{op:descents} for $\sigma$, the swaps start at the last element $z\in \sigma_j$ and some non-last element $\pm (z+\epsilon)\in \sigma_{k}$ with $k>j$. If the swaps terminate without a swap involving $\sigma_{k,1}$, then the last element in $\sigma_{k}$ cannot be involved in a swap. Hence, we must analyze when a swap involving $\sigma_{k,1}$ occurs. In this case, \cref{cor:plus1} implies that $\sigma_{k,1}$ must swap with an element of the opposite sign, so \cref{lem:continued_swaps} implies that the swaps will terminate at this step. Hence, the algorithm will not proceed into swapping the last element of $\sigma_{k}$. As the last element of $\sigma_k$ is not involved in a swap, neither can any element right of $\pm (z+\epsilon)$ in the cycle notation of $\sigma$, so swap property \ref{sp:last} holds for $\sigma$.

    We now consider order property \ref{op:first} for $\sigma'$. Order property \ref{op:first} holds for $\sigma$, so the first elements in $\sigma_1,\sigma_2,\ldots,\sigma_m$ are in increasing order. By \cref{lem:SPOP1}, swaps only occur between $\sigma_j$ and some $\sigma_{k}$ with $k>j$, so consider any swaps that affect the first elements in these cycles, which are $\sigma_{j,1}$ and $\sigma_{k,1}$. If no swaps occur involving these elements, then it follows that the first elements in $\sigma_1',\sigma_2',\ldots,\sigma_m'$ are also in increasing order. If a swap involves only $\sigma_{j,1}$ (resp. $\sigma_{k,1}$), then the first element in cycle $j$ (resp. $k$) changes by $1$, while all other remaining first elements are unchanged. This also cannot affect the relative order between any of the first elements in the cycles. If a swap involves both $\sigma_{j,1}$ and $\sigma_{k,1}$, then \cref{cor:plus1} implies that $\sigma_{j,1}$ and $\sigma_{k,1}$ have opposite signs. A swap changes these elements of opposite signs by $1$ each, which again cannot affect the relative order between any of the first elements. In all cases, the first elements in $\sigma_1',\sigma_2',\ldots,\sigma_m'$ are in increasing order, which is the first claim in order property \ref{op:first}. \cref{lem:SPOP1} implies that the largest element in each cycle of $\sigma'$ is the first one. Combined, we see that for each $i\in [m]$, $\sigma_{i,1}'$ is the larger than any element in $\sigma_1',\ldots,\sigma_i'$, so order property \ref{op:first} holds for $\sigma'$. 
\end{proof}

\begin{lemma}\label{lem:SP3}
Suppose the order properties hold for $\sigma$, and the order and swap properties hold for all signed permutations prior to $\sigma$ in the algorithm. Then swap property \ref{sp:large} holds for $\sigma$. 
\end{lemma}

\begin{proof}
Let $x$ be an element appearing in the cycle notation of $\sigma$ that satisfies $\sigma(|x|)\geq \sigma_{j+1,1}$. Order property \ref{op:first} for $\sigma$ implies that $x$ must be in one of the cycles $\sigma_{j+1},\ldots,\sigma_k$. Furthermore, if $x$ is the last element in the cycle that contains it, then \cref{lem:SPOP2} implies that $x$ cannot be involved in any swaps. Hence, it suffices to consider when $x$ appears as a non-last element in some cycle $\sigma_k$, where $k\geq j+1$. As $\sigma(|x|)=\sigma_{j+1,1}$ can only occur if $x$ is the last element of $\sigma_{j+1}$, we can assume $k>j+1$, so $\sigma(|x|)>\sigma_{j+1,1}$. We will consider cases based on iteration of the \texttt{while} loop in line \ref{line:whileloop2}.

We first show by contradiction that $x$ cannot be involved in the first swap of the \texttt{while} loop. For this to occur, order property \ref{op:descents} for $\sigma$ implies that $x=\pm (z+\epsilon)$, where $z$ is the last element of $\sigma_j$, $\epsilon\in \{-1,1\}$, and $\min\{|x|,|z|\}\in \Des_{\Delta}(\pi,\sigma)$. By order property \ref{op:descents} for $\sigma$, we have that
\[\sigma(|x|)> \sigma_{j+1,1}>\sigma_{j,1}=\sigma(|z|)\]
and $\pi(|x|)<\pi(|z|)$. As $\sigma(|x|)> \sigma_{j+1,1}$, order property \ref{op:large} for $\sigma$ implies that $\pi(|x|)>\pi_{j+1,1}$, so combined, we have $\pi(|z|)>\pi(|x|)>\pi_{j+1,1}$. However, order property \ref{op:large} for $\sigma$ then implies that $z$ cannot have been involved in swaps. As $z$ is the last element in $\sigma_j$, it must also be the last element in $\pi_j$. Then $\pi(|z|)=\pi_{j+1,1}$, which contradicts $\pi(|z|)>\pi(|x|)>\pi_{j+1,1}$.

We next consider the case where $x$ is not the first swap in this iteration of the \texttt{while} loop. For $x$ to be involved in a swap, the element following $x$, which is $\sigma(|x|)$, must be swapped first. As $\sigma(|x|)>\sigma_{j+1,1}>\sigma_{j,1}$, order property \ref{op:first} for $\sigma$ implies that $\sigma(|x|)\pm 1\notin \sigma_j$. Then $\sigma(|x|)$ can only being swapped with $-\sigma(|x|)\pm 1$. By \cref{lem:continued_swaps}, the \texttt{while} loop will terminate after swapping these elements of opposite signs, so $x$ itself will not be involved in a swap. In all cases, $x$ cannot be involved in a swap, so swap property \ref{sp:large} holds for $\sigma$. 
\end{proof}

\begin{lemma}\label{lem:OP3}
Suppose the order properties hold for $\sigma$, and the order and swap properties hold for all signed permutations prior to $\sigma$ in the algorithm. Then order property \ref{op:large}  holds for $\sigma'$. 
\end{lemma}

\begin{proof}
We start with order property \ref{op:large} for $\sigma$, which states that for $k\geq j+1$, we have $\pi(|x|)> \pi_{k,1}$ if and only if $\sigma(x)\geq \sigma_{k,1}$, where equality holds only when $x$ is the last element in $\sigma_k$. Swap property \ref{sp:large} for $\sigma$ holds by \cref{lem:SP3}, so this implies that any such $x$ will not be involved in any swaps. If $x$ is the last element of $\sigma_k$, then it is clear that $\sigma'(|x|)=\sigma_{k,1}'$. Otherwise, $\sigma(|x|)$ and $\sigma_{k,1}$ are distinct elements.  A swap cannot simultaneously involve both $\sigma_{k,1}$ and $\sigma(|x|)$ by \cref{lem:SPOP1}, as neither element is in $\sigma_j$. Additionally, a swap that changes only one of $\sigma_{k,1}$ or the element following $x$ does not affect their relative order. Consequently, we see that $\sigma'(|x|)\geq \sigma'_{k,1}$ holds. Similar reasoning shows that swaps cannot introduce a new element $x'$ satisfying $\sigma'(|x'|)\geq\sigma'_{k,1}$ without swapping some $x$ satisfying $\sigma(|x|)\geq \sigma_{k,1}$, so $\sigma(|x|)\geq \sigma_{k,1}$ if and only if $\sigma'(|x|)\geq \sigma'_{k,1}$, where it is clear that equality can only hold when $x$ is the last element of $\sigma_k$ and $\sigma_k'$. In the case where the \texttt{for} loop remains in iteration $j$ after swaps are performed, we conclude order property \ref{op:large} for $\sigma'$. In the case where the \texttt{for} loop proceeds to iteration $j+1$, then order property \ref{op:large} follows by specializing to $k\geq j+2$. 
\end{proof}

Finally, we turn our attention to swap property \ref{sp:descents} for $\sigma$ and order property \ref{op:descents} for $\sigma'$. We will prove various parts of swap property \ref{sp:descents} in separate lemmas.

\begin{lemma}\label{lem:removeddescent}
    Suppose the order properties hold for $\sigma$, and the swap properties hold for all signed permutations prior to $\sigma$ in the algorithm. Let $z$ be the last element in $\sigma_j$, and let $\pm (z+\epsilon)$ with $\epsilon\in \{-1,1\}$ be the element that it is swapped with. The swaps in the algorithm remove the descent $\min\{|z|,|z+\epsilon|\}\in \Des_{\Delta}(\pi,\sigma)$. 
\end{lemma}

\begin{proof}
    From order property \ref{op:descents} for $\sigma$, swaps begin due to $\min\{|z|,|z+\epsilon|\}\in \Des_{\Delta}(\pi,\sigma)$, where $ z\in \sigma_j$ is the last element and $\pm (z+\epsilon)$ is a non-last element in some $\sigma_{k}$ with $k>j$. The first swap interchanging $z$ with $\pm(z+\epsilon)$ changes the relative order of the elements following them in the cycle notation, which removes the descent $\min\{|z|,|z+\epsilon|\}\in \Des_{\Delta}(\pi,\sigma)$.  \cref{lem:SPOP2} implies that the element now following $\pm z$ in the cycle notation does not change from the later swaps.  While these later swaps can involve the element following $\pm (z+\epsilon)$,  this element changes by at most $1$, so the relative order of what follows $z$ and $\pm(z+\epsilon)$ does not change after the first swap. We conclude the descent $\min\{|z|,|z+\epsilon|\}\in \Des_{\Delta}(\pi,\sigma)$ is removed. 
\end{proof}

\begin{lemma}\label{lem:epsilon_orders}
    Suppose the order properties hold for $\sigma$, and the order and swap properties hold for all signed permutations prior to $\sigma$ in the algorithm. Let $z$ be the last element in $\sigma_j$, let $\pm (z+\epsilon)$ with $\epsilon\in \{-1,1\}$ be the element that it is swapped with, and suppose $\min\{|z|,|z-\epsilon|\}\in [n-1]$. If $\min\{|z|,|z+\epsilon|\}\in \Des_{\Delta}(\pi,\sigma)$ and $\min\{|z|,|z-\epsilon|\}\notin \Des_{\Delta}(\pi,\sigma)$, then $\sigma(|z|)<\sigma(|z-\epsilon|)<\sigma(|z+\epsilon|)$ cannot occur. 
\end{lemma}
\begin{proof}
    If $\sigma$ is the signed permutation at the start of Algorithm \ref{algorithm}, then this follows from \cref{cor:algstart}. Otherwise, some swaps have occurred, and we will prove the result in this case by contradiction. Assume $\sigma(|z|)<\sigma(|z-\epsilon|)<\sigma(|z+\epsilon|)$ holds. Since $\min\{|z|,|z+\epsilon|\}\in \Des_{\Delta}(\pi,\sigma)$ and $\min\{|z|,|z-\epsilon|\}\notin \Des_{\Delta}(\pi,\sigma)$, it must be that
    \begin{equation}\label{eq:epsilon_orders}
        \pi(|z-\epsilon|)>\pi(|z|)>\pi(|z+\epsilon|).
    \end{equation}
    Now let $\tilde{\sigma}=\tilde{\sigma}_1\tilde{\sigma}_2\dots \tilde{\sigma}_m$ be the signed permutation prior to the last sequence of swaps, so in a previous iteration of the \texttt{while} loop, swaps on $\tilde{\sigma}$ resulted in $\sigma$. We consider two cases based on the iteration of the \texttt{for} loop. 

    First, suppose that the \texttt{for} loop was not in iteration $j$ when swaps are performed on $\tilde{\sigma}$. Then swap property \ref{sp:last} for permutations prior to $\sigma$ implies that the last element of cycle $j$ has not been involved in any swaps, and since this is $ z$ in $\sigma$, it must have been $ z$ throughout the algorithm. In particular, \eqref{eq:epsilon_orders} becomes
    \begin{equation}\label{eq:epsilon_orders2}
        \pi(|z-\epsilon|)>\pi(|z|)=\pi_{j+1,1}>\pi(|z+\epsilon|).
    \end{equation}
    From order property \ref{op:large} for $\sigma$, we conclude $\sigma(|z-\epsilon|)>\sigma_{j+1,1}$. However,  combined with $\sigma(|z+\epsilon|)>\sigma(|z-\epsilon|)$, this implies $\sigma(|z+\epsilon|)>\sigma_{j+1,1}$. Comparing with \eqref{eq:epsilon_orders2}, we see that $\pm (z+\epsilon)$ is an element contradicting order property \ref{op:large} for $\sigma$.

    Alternatively, suppose the \texttt{for} loop was in iteration $j$ when swaps are performed on $\tilde{\sigma}$. We begin by considering the positions of $z$, $\pm (z-\epsilon)$, and $\pm (z+\epsilon)$ in the cycles of $\tilde{\sigma}$ containing them. As $z$ is the last element of $\sigma_j$, this element must have been swapped into this position during the swaps on $\tilde{\sigma}$. By construction of the algorithm, $\pm z$ must have been swapped with $\pm (z-\epsilon)$ due to $\min\{|z|,|z-\epsilon|\}\in \Des_{\Delta}(\pi,\tilde{\sigma})$. In particular, order property \ref{op:descents} for $\tilde{\sigma}$ implies that $\pm (z-\epsilon)$ is the last element of $\tilde{\sigma}_j$ and $\pm z$ is a non-last of a cycle to the right of $\tilde{\sigma}_j$. Order property \ref{op:descents} for $\sigma$ with $\min\{|z|,|z+\epsilon|\}\in \Des_{\Delta}(\pi,\sigma)$ implies that $\pm (z+\epsilon)$ is a non-last element in $\sigma_{j+1},\ldots,\sigma_m$. Then this must have been a non-last element in $\tilde{\sigma}_j,\tilde{\sigma}_{j+1},\ldots,\tilde{\sigma}_{j+1}$ by swap properties \ref{sp:right} and \ref{sp:last} for $\tilde{\sigma}$. 
    
   Combined, we see that $\pm (z-\epsilon)$ is the last element in $\tilde{\sigma}_j$, and both $\pm z$ and $\pm (z+\epsilon)$ are non-last elements in the cycles of $\tilde{\sigma}$ containing them. In particular, order property \ref{op:descents} for $\tilde{\sigma}$ implies that $\min\{|z|,|z+\epsilon|\}\notin \Des_{\Delta}(\pi,\tilde{\sigma})$. The swaps on $\tilde{\sigma}$ must introduce the descent $\min\{|z|,|z+\epsilon|\}\in \Des_{\Delta}(\pi,\sigma)$, so swap property \ref{sp:descents} for $\tilde{\sigma}$ implies that $\tilde{\sigma}(|z+\epsilon|)<\tilde{\sigma}(|z|)$. As order property \ref{op:descents} for $\tilde{\sigma}$ implies $\tilde{\sigma}(|z-\epsilon|)<\tilde{\sigma}(|z|)$, we see that swapping $\pm (z-\epsilon)$ and $\pm z$ results in the element following $\pm (z-\epsilon)$ being larger than the one following $\pm (z+\epsilon)$. \cref{lem:SPOP2} implies that swaps do not affect the element following $\pm (z-\epsilon)$ in the cycle notation, while they can only change what follows $\pm (z+\epsilon)$ by $1$. Hence, we conclude $\sigma(|z+\epsilon|)<\sigma(|z-\epsilon|)$, and this contradicts $\sigma(|z|)<\sigma(|z-\epsilon|)<\sigma(|z+\epsilon|)$.
\end{proof}

\begin{lemma}\label{lem:removeddescent2}
    Suppose the order properties hold for $\sigma$, and the swap properties hold for all signed permutations prior to $\sigma$ in the algorithm. Let $z$ be the last element in $\sigma_j$, let $\pm (z+\epsilon)$ with $\epsilon\in \{-1,1\}$ be the element that it is swapped with, and suppose $\min\{|z|,|z-\epsilon|\}\in [n-1]$. 
    \begin{enumerate}[label=(\alph*)]
        \item If $\min\{|z|,|z-\epsilon|\}\in \Des_{\Delta}(\pi,\sigma)$, then swaps remove this descent from the symmetric difference.
        \item If $\min\{|z|,|z-\epsilon|\}\notin \Des_{\Delta}(\pi,\sigma)$, then swaps do not introduce this to the symmetric difference.
    \end{enumerate}
\end{lemma}

\begin{proof}
    We first consider (a). In this case, we have that $\min\{|z|,|z+\epsilon|\}$ and $\min\{|z|,|z-\epsilon|\}$ are both elements of $\Des_{\Delta}(\pi,\sigma)$. Order property \ref{op:descents} for $\sigma$ implies that  $\pi(|z|)>\pi(|z+\epsilon|)$, $\sigma(|z|)<\sigma(|z+\epsilon|)$, $\pi(|z|)>\pi(|z-\epsilon|)$, and $\sigma(|z|)<\sigma(|z-\epsilon|)$. Algorithm \ref{algorithm} chooses to swap $z$ with $\pm (z+\epsilon)$ instead of $\pm (z-\epsilon)$, and this can only occur if $\sigma(|z+\epsilon|)>\sigma(|z-\epsilon|)>\sigma(|z|)$. Consequently, swapping $ z$ with $\pm (z+\epsilon)$ also removes the descent $\min\{|z|,|z-\epsilon|\}$. Using a similar argument as in \cref{lem:removeddescent} and \cref{lem:epsilon_orders}, the status of this descent is not affected by the remaining swaps in this iteration of the \texttt{while} loop in line \ref{line:whileloop2}. 
    
    For (b), we have two cases for $\min\{|z|,|z-\epsilon|\}\notin \Des_{\Delta}(\pi,\sigma)$. The first scenario is $\pi(|z|)>\pi(|z-\epsilon|)$ and $\sigma(|z|)>\sigma(|z-\epsilon|)$, and in this case, order property \ref{op:descents} implies $\sigma(|z+\epsilon|)>\sigma(|z|)>\sigma(|z-\epsilon|)$. Then swapping $z$ with $\pm(z+\epsilon)$ does not introduce the descent $\min\{|z|,|z-\epsilon|\}$ into the symmetric difference. As in earlier arguments, subsequent swaps also do not. The second scenario is $\pi(|z|)<\pi(|z-\epsilon|)$ and $\sigma(|z|)<\sigma(|z-\epsilon|)$. Order property \ref{op:descents} for $\sigma$ implies that  $\pi(|z|)>\pi(|z+\epsilon|)$ and $\sigma(|z|)<\sigma(|z+\epsilon|)$, and combined with \cref{lem:epsilon_orders}, we conclude that $\sigma(|z|)<\sigma(|z+\epsilon|)<\sigma(|z-\epsilon|)$. A similar argument shows that swaps again do not introduce $\min\{|z|,|z-\epsilon|\}$ into the symmetric difference. 
\end{proof}

\begin{lemma}\label{lem:removeddescent3}
    Suppose the order properties hold for $\sigma$, and the swap properties hold for all signed permutations prior to $\sigma$ in the algorithm. Let $z$ be the last element in $\sigma_j$, let $\pm (z+\epsilon)$ with $\epsilon\in \{-1,1\}$ be the element that it is swapped with, and suppose $\min\{|z+\epsilon|,|z+2\epsilon|\}\in [n-1]$. 
    \begin{enumerate}[label=(\alph*)]
        \item If $\min\{|z+\epsilon|,|z+2\epsilon|\}\in \Des_{\Delta}(\pi,\sigma)$, then this descent is removed. 
        \item If $\min\{|z+\epsilon|,|z+2\epsilon|\}\notin \Des_{\Delta}(\pi,\sigma)$ and $\sigma(|z+2\epsilon|)>\sigma(|z+\epsilon|)$, then this descent is not introduced to the symmetric difference
    \end{enumerate}
\end{lemma}

\begin{proof}
    For (a), both $\min\{|z|,|z+\epsilon|\}$ and $\min\{|z+\epsilon|,|z+2\epsilon|\}$ are elements of $\Des_{\Delta}(\pi,\sigma)$. By order property \ref{op:descents} for $\sigma$, this can only occur if $\pm (z+\epsilon)$ is a non-last element in some cycle $\sigma_k$ with $k>j$, and combined with order property \ref{op:first}, we also conclude that $\pm (z+2\epsilon)$ must be the last element in some cycle in $\sigma_{j+1},\ldots,\sigma_{k-1}$. Additionally, the combination of these properties also implies $\sigma(|z|)<\sigma(|z+2\epsilon|)<\sigma(|z+\epsilon|)$. The first swap between $z$ and $\pm (z+\epsilon)$ will remove $\min\{|z+\epsilon|,|z+2\epsilon|\}\in \Des_{\Delta}(\pi,\sigma)$. Later swaps in the \texttt{while} loop can only change the element following $\pm (z+\epsilon)$ by at most $1$ and do not affect the cycle containing $\pm (z+2\epsilon)$. Consequently, $\min\{|z+\epsilon|,|z+2\epsilon|\}$ cannot be re-introduced into the symmetric difference, so (a) follows. 

    For (b), order property \ref{op:descents} for $\sigma$ and the assumptions imply $\sigma(|z|)<\sigma(|z+\epsilon|)<\sigma(|z+2\epsilon|)$, where $\pm (z+\epsilon)$ appears as a non-last element in $\sigma_{j+1},\ldots,\sigma_{m}$. We let $\tilde{\sigma}$ be the signed permutation obtained by swapping $z$ and $\pm (z+\epsilon)$, which satisfies $\sigma(|z+\epsilon|)=\tilde{\sigma}(|z|)$. The relative orders of this element and what follows $z$ and $\pm (z+2\epsilon)$ in the cycle notation of $\sigma$ is preserved by this first swap, so $\tilde{\sigma}(|z+\epsilon|)<\tilde{\sigma}(|z|)<\tilde{\sigma}(|z+2\epsilon|)$. As $\tilde{\sigma}(|z+\epsilon|)\neq \tilde{\sigma}(|z+2\epsilon|)-1$, later swaps cannot simultaneously involve the elements following $z+\epsilon$ and $\pm (z+2\epsilon)$ in the cycle notation of $\tilde{\sigma}$. Consequently, the relative order of these elements is preserved, and $\min\{|z+\epsilon|,|z+2\epsilon|\}$ is not introduced into the symmetric difference. We conclude that (b) holds.
\end{proof}

\begin{lemma}\label{lem:SP4}
    Suppose the order properties hold for $\sigma$, and the swap properties hold for all signed permutations prior to $\sigma$ in the algorithm.  Then swap property \ref{sp:descents} holds for $\sigma$. 
\end{lemma}

\begin{proof}
    The first three claims of swap property \ref{sp:descents} for $\sigma$ involve $\min\{|z|,|z+\epsilon|\}$, $\min\{|z|,|z-\epsilon|\}$, and $\min\{|z+\epsilon|,|z+2\epsilon|\}$, and the results for these are given in \cref{lem:removeddescent,lem:removeddescent2,lem:removeddescent3}. For all other descents, consider the effect on the one-line notation when $z$ and $\pm (z+\epsilon)$ are swapped. In addition to affecting the images of $|z|$ and $|z+\epsilon|$, this changes the images of $|\sigma^{-1}(z)|$ and $|\sigma^{-1}(z+\epsilon)|$ by $1$. From the one-line notation of $\sigma$, this cannot introduce other descents into the symmetric difference unless $|\sigma^{-1}(z)|$ and $|\sigma^{-1}(z+\epsilon)|$ are consecutive elements of $[n]$. In this case, the algorithm then swaps the elements proceeding $z$ and $\pm(z+\epsilon)$ to remove this descent. Iterating this argument, swaps continue until one of the three cases in \cref{lem:continued_swaps} occurs. If swaps end due to the first element of $\sigma_j$ being swapped, then this can only affect descents that involve $\pm (z+\epsilon)$, and these are accounted for in \cref{lem:removeddescent,lem:removeddescent3}. In the remaining two cases, the final swap does not introduce any descent, so we conclude swap property \ref{sp:descents} holds for $\sigma$. 
\end{proof}

\begin{lemma}\label{lem:OP4}
    Suppose the order properties hold for $\sigma$, and the swap properties hold for all signed permutations prior to $\sigma$ in the algorithm.  Then order property \ref{op:descents} holds for $\sigma'$. 
\end{lemma}

\begin{proof}
    We begin with order property \ref{op:descents} for $\sigma$, and apply the swap properties for $\sigma$, which hold by \cref{lem:SPOP1,lem:SPOP2,lem:SP3,lem:SP4}. Swap property \ref{sp:descents} implies that $\min\{|z|,|z-\epsilon|\}$ and $\min\{|z|,|z+\epsilon|\}$ cannot be elements of $\Des_{\Delta}(\pi,\sigma')$, and descents other than $\min\{|z+\epsilon|,|z+2\epsilon|\}$ cannot be introduced. We will begin by considering descents different from these three aforementioned ones and then consider possibilities for $\min\{|z+\epsilon|,|z+2\epsilon|\}$.
    
    Let $d\in \Des_{\Delta}(\pi,\sigma)$ be a descent that is not $\min\{|z|,|z+\epsilon|\}$, $\min\{|z|,|z-\epsilon|\}$, or $\min\{|z+\epsilon|,|z+2\epsilon|\}$. Using order property \ref{op:descents} for $\sigma$, we let $x\in \{\pm d,\pm (d+1)\}$ be whichever element appears as the last element of a cycle in $\sigma_{j+1},\ldots,\sigma_m$ and $y\in \{\pm d,\pm (d+1)\}$ be the element that appears in the non-last position of a cycle to its right. Combined with order property \ref{op:first}, we also have that $\sigma_{j+1,1}\leq \sigma(|x|)<\sigma(|y|)$ while $\pi(|x|)>\pi(|y|)$. By swap property \ref{sp:right} for $\sigma$, the cycles containing $x$ and $y$ cannot both be affected by swaps, and by swap property \ref{sp:large}, $x$ and $y$ themselves are not affected by swaps. Hence, swaps preserve the positions of $x$ and $y$ in the cycle notation and change the element following at most one of them by $1$, resulting in $\sigma'(|x|)<\sigma'(|y|)$. We conclude that the claims in order property \ref{op:descents} involving $d\in \Des_{\Delta}(\pi,\sigma')$ hold.

    Finally, we consider $\min\{|z+\epsilon|,|z+2\epsilon|\}$. If $\min\{|z+\epsilon|,|z+2\epsilon|\}\notin \Des_{\Delta}(\pi,\sigma')$, then the \texttt{for} loop terminates iteration $j$, and the last element of $\sigma_{j}'$, which is $\pm (z+\epsilon)$, is not involved in any descents in $\Des_{\Delta}(\pi,\sigma')$. In this case, we conclude order property \ref{op:descents} holds for $\sigma'$. It now suffices to consider when swaps result in $\min\{|z+\epsilon|,|z+2\epsilon|\}\in \Des_{\Delta}(\pi,\sigma')$. By swap property \ref{sp:descents} for $\sigma$, this can only occur if $\min\{|z+\epsilon|,|z+2\epsilon|\}\notin \Des_{\Delta}(\pi,\sigma)$ and $\sigma(|z+2\epsilon|)<\sigma(|z+\epsilon|)$, which combined also imply $\pi(|z+2\epsilon|)<\pi(|z+\epsilon|)$. We now rule out cases for where $\pm (z+2\epsilon)$ can appear in $\sigma'$.
    \begin{itemize}
        \item As $\pm (z+\epsilon)$ is involved in the swaps on $\sigma$, swap property \ref{sp:large} for $\sigma$ shows that $\sigma(|z+2\epsilon|)<\sigma(|z+\epsilon|)<\sigma_{j+1,1}.$ Then order property \ref{op:first} for $\sigma$ implies that $\pm (z+2\epsilon)$ cannot be the last element in $\sigma_{j+1},\ldots,\sigma_m$, and swap property \ref{sp:last} for $\sigma$ implies that swaps cannot result in $\pm (z+2\epsilon)$ becoming the last element in $\sigma_{j+1}',\ldots,\sigma_m'$.
        \item If $\pm (z+2\epsilon)$ is in $\sigma_1',\sigma_2',\ldots,\sigma_j'$, then \cref{lem:SPOP2} implies that $\sigma'(|z+2\epsilon|)<\sigma_{j,1}'=\sigma'(|z+\epsilon|)$. Since this order aligns with $\pi(|z+2\epsilon|)<\pi(|z+\epsilon|)$, this would imply $\min\{|z+\epsilon|,|z+2\epsilon|\}\notin \Des_{\Delta}(\pi,\sigma)$. 
    \end{itemize}
    Combined, we see that $\min\{|z+\epsilon|,|z+2\epsilon|\}\in \Des_{\Delta}(\pi,\sigma')$ can only occur if $\pm(z+2\epsilon)$ is a non-last element in $\sigma_{j+1}',\ldots,\sigma_m'$. Furthermore, it must be that $\pi(|z+2\epsilon|)<\pi(|z+\epsilon|)$ and $\sigma'(|z+2\epsilon|)>\sigma'(|z+\epsilon|)$, so order property \ref{op:descents} holds for $\sigma'$.
\end{proof}

\begin{theorem}\label{thm:order_swap}
    Let $\sigma$ be the signed permutation at the start of an iteration of the \texttt{for} loop or the \texttt{while} loop in line \ref{line:whileloop1} of Algorithm \ref{algorithm}. Then the order properties and swap properties hold for $\sigma$. 
\end{theorem}

\begin{proof}
    We prove this using strong induction. \cref{lem:algstart} implies that the order properties hold for the permutation $\sigma$ at the start of the algorithm. Additionally, the swap properties vacuously hold for all permutations prior to $\sigma$ in the algorithm. Assuming the order properties hold for $\sigma$, and the order and swap properties held prior to $\sigma$, \cref{lem:no_swaps,lem:SPOP1,lem:SPOP2,lem:SP3,lem:OP3,lem:SP4,lem:OP4} imply that the swap properties hold for $\sigma$ and the order properties will hold for the signed permutation at the next iteration of the \texttt{for} or \texttt{while} loop. Hence, by strong induction, these properties hold throughout the algorithm.
\end{proof}

\begin{corollary}\label{lem:descents_preserved}
    Let $\pi\in \C_{B,n+1}^+$. Then $\Des(\pi)\cap [n-1]=\Des(\phi(\pi))$. 
\end{corollary}

\begin{proof}
    Let $\sigma$ be the signed permutation in iteration $m$ of the \texttt{for} loop. By \cref{thm:order_swap}, order property \ref{op:descents} holds for $\sigma$. Then $d\in \Des_{\Delta}(\pi,\sigma)\cap [n-1]$ can only exist if $\pm d$ is the last element of $\sigma_m$ and $\pm (d+1)$ is in a cycle to its right, but there are no such cycles. We conclude $\Des_{\Delta}(\pi,\sigma)\cap [n-1]=\emptyset$, and this last iteration of the \texttt{for} loop performs no swaps. From this, we see that $\phi(\pi)=\sigma$ and $\Des(\pi)\cap [n-1]=\Des(\phi(\pi))$.
\end{proof}

\begin{corollary}\label{cor:canonical}
    Let $\sigma$ be the signed permutation at the start of an iteration of the \texttt{for} loop or the \texttt{while} loop in line \ref{line:whileloop1} of Algorithm \ref{algorithm}. Then $\sigma$ is in canonical cycle notation.
\end{corollary}

\begin{proof}
    Observe that the canonical cycle notation properties are given by order property \ref{op:first}. Hence, this follows immediately from \cref{thm:order_swap}.
\end{proof}

\cref{lem:descents_preserved} shows that applying $\phi$ to $\pi\in \C_{B,n+1}^+$ will result in a signed permutation with the same descents in $[n-1]$. As demonstrated in \cref{ex:phi}, it is possible that $0\in \Des_{\Delta}(\pi,\phi(\pi))$, and we will now address the descent at position $0$. For the remainder of this section, we will use the order and swap properties freely without explicitly citing \cref{thm:order_swap}.

\begin{lemma}\label{lem:nodescent0preserved}
    If $0\in \Des(\pi)$, then $0\in \Des(\sigma)$. In this case, we have that $\pi(1)<-1$ and $\sigma(1)<-1$. 
\end{lemma}

\begin{proof}
    We will use a proof by induction, beginning with the case where $\sigma$ is the initial permutation in the algorithm. As $\pi$ is a cyclic signed permutation, $\pi(1)<0$ can only occur if $\pi(1)<-1$. Then $\pi(1)$ cannot be a left-to-right maximum, so by the construction of $\sigma$ using left-to-right maxima, we conclude $\sigma(1)=\pi(1)<-1$.

    For our induction step, assume that $\sigma(1)<-1$, and let $\sigma'$ be the signed permutation after swaps in the \texttt{while} loop in line \ref{line:whileloop2}. If no swaps involving $\pm 1$ or the element following it occur, then $\sigma'(1)=\sigma(1)<-1$. Hence, we consider the remaining cases where swaps do occur. 

    We first prove by contradiction that swaps cannot begin with $\pm 1$. If this occurs, $\pm 1$ must be swapped with $\pm 2$. As $\sigma(1)<-1$, order property \ref{op:first} implies that $\pm 1$ cannot be the last element in the cycle containing it. Then order property \ref{op:descents} implies that $\pm 2$ appears as the last element of some cycle to the left of $\pm 1$ and $\sigma(2)<\sigma(1)<-1$. As $\sigma(1)\notin \{-1,-2\}$, we find that $\sigma(1)\leq -3$ and $\sigma(2)<-3$. However, $\pm 2$ is the last element in the cycle containing it, so $\sigma(2)<-2$ is the first element of that cycle, and this contradicts order property \ref{op:first}.

    It now suffices to consider when the element following $\pm 1$ is swapped first. The sign of the element following $\pm 1$ does not change, so if swaps terminate after this, then $\sigma'(1)<-1$. Otherwise, $\pm 1$ must then be swapped with $\pm 2$. By \cref{lem:continued_swaps}, this can only occur if the elements following $\pm 1$ and $\pm 2$ in $\sigma$ share the same sign, and hence must both be negative. We see that even in the case when $\pm 1$ and $\pm 2$ are swapped, we have $\sigma'(1)<-1$. 
\end{proof}

\begin{corollary}\label{cor:nodescent0preserved}
     If $0\in \Des(\pi)$, then $0\in \Des(\phi(\pi))$. Furthermore, in this case, we have that $\pi(1)<-1$ and $\phi(\pi)(1)<-1$. 
\end{corollary}

\begin{lemma}\label{lem:descent0}
    If $0\notin \Des(\pi)$, then either $0\notin \Des(\sigma)$ or $\sigma(1)=-1$.
\end{lemma}

\begin{proof}
    We will prove this using induction, again starting with the case where ${\sigma}$ is the signed permutation at the start of the algorithm. If $\pm 1$ is not the last element of a cycle in ${\sigma}$, then ${\sigma}(1)=\pi(1)$, so $0\notin \Des(\pi)$ implies $0\notin \Des(\sigma)$. Otherwise, ${\sigma}(1)$ is the first element of the cycle containing $\pm 1$, and order property \ref{op:first} implies that $\sigma(1)\geq -1$, so either $\sigma(1)=-1$ or $\sigma(1)\geq 1$. The latter case is equivalent to $0\notin \Des(\sigma)$.

    Now assume that the result holds for $\sigma$, and let $\sigma'$ be the permutation obtained after swaps in the \texttt{while} loop in line \ref{line:whileloop2}. As in \cref{lem:nodescent0preserved}, we can assume that some swap involves $\pm 1$ or the element following it in the cycle notation. Our argument will require cases based on which property from the induction hypothesis holds for $\sigma$.
    
    First, consider the induction hypotheses of $0\notin \Des(\sigma)$. We have two cases based on whether or not $\pm 1$ is involved in the first swap. If it is not, then the element following $\pm 1$, denoted $x$, is swapped first. As $0\notin \Des(\sigma)$, we know $x>0$. If swaps terminate after $x$ is swapped, then the sign of the element following $\pm 1$ has not changed, so $\sigma'(1)>0$. Otherwise, $\pm 1$ is then swapped with $\pm 2$, and the result is $\sigma'(1)=x>0$. We conclude that $0\notin \Des(\sigma')$ holds in each case. 
    
    Alternatively, suppose $\pm 1$ is involved in the first swap. Then this occurs due to $1\in \Des_{\Delta}(\pi,\sigma)$, and order property \ref{op:descents} for $\sigma$ implies that either $\pm 1$ or $\pm 2$ is the last element of the cycle containing it. If $\pm 1$ is the last element of a cycle, then order property \ref{op:descents} also implies $\pm 2$ appears in a cycle to the right of $\pm 1$ and $\sigma(2)>\sigma(1)>0$. Then the swap between $\pm 1$ and $\pm 2$ results in a positive number following $1$ in the cycle notation. By \cref{lem:SPOP2}, subsequent swaps cannot change this value, so $\sigma'(1)=\sigma(2)>0$ and $0\notin \Des(\sigma')$. If $\pm 2$ is the last element of a cycle, then order property \ref{op:first} implies $\sigma(2)\geq -2$. Combined with the fact that $\pm 1$ is not in the same cycle as $\pm 2$, we conclude that either $(-2)$ is a cycle in $\sigma$ or $\sigma(2)\geq 2$. In the first case, $(-1)$ will be a cycle in $\sigma'$, and in the second case, $\sigma'(1)\geq 1$. These respectively correspond to $\sigma'(1)=-1$ or $0\notin \Des(\sigma')$.

    Finally, consider the induction hypothesis that $(-1)$ is a cycle in $\sigma$. In this case, a swap involving $-1$ can only occur if it is swapped with $\pm 2$ in the first and only swap performed on $\sigma$. Order property \ref{op:descents} implies that $\pm 2$ appears as a non-last element in a cycle to the right of $(-1)$, and $\sigma(1)=-1<\sigma(2)$. In particular, we see that $\sigma(2)>0$, and the single swap in this \texttt{while} loop will then result in $\sigma'(1)=\sigma(2)>0$. Hence $0\notin \Des(\sigma')$. By induction, we conclude that the result holds for all $\sigma$ encountered in the algorithm. 
\end{proof}

\begin{corollary}\label{cor:descent0}
    We have that $0\in \Des_{\Delta}(\pi,\phi(\pi))$ if and only if $\pi(1)>0$ and $\phi(\pi)(1)=-1$. 
\end{corollary}

\begin{proof}
    It is clear that if $\pi(1)>0$ and $\phi(\pi)(1)=-1$, then $0\in \Des_{\Delta}(\pi,\phi(\pi))$, so it suffices to show the converse.
    \cref{cor:nodescent0preserved} shows that $0\in \Des_{\Delta}(\pi,\phi(\pi))$ can only occur if $0\notin \Des(\pi)$ and $0\in \Des(\phi(\pi))$. Equivalently, $\pi(1)>0$ but $\phi(\pi)(1)<0$. By \cref{lem:descent0}, this can only occur if $\pi(1)>0$ and $\phi(\pi)(1)=-1$.
\end{proof}

\subsection{The descent-preserving function}\label{sec:descent_preserving_map}

Using the results from the preceding section, we are now prepared to define our descent-preserving function in \cref{thm:main_thm}. The basis of this function is several variations of $\phi$. There will be four separate cases depending on whether $\pi\in \C_{B,n+1}^+$ or $\pi\in \C_{B,n+1}^-$, and whether or not the descent $0$ is introduced by Algorithm \ref{algorithm}.

\begin{definition}
    For any $\pi\in \C_{B,n+1}$, define $-\pi=[-1,-2,\ldots,-n]  \pi$, which is the permutation obtained by changing the signs of all elements in the one-line or cycle notation. Define the function $\Phi:\C_{B,n+1}\to B_n$ by
    \[\Phi(\pi) = \begin{cases} \phi(\pi) & \text{ if $\pi \in \C_{B,n+1}^+$ and $0\notin \Des_{\Delta}(\pi,\phi(\pi))$}  \\
     (-1)  \phi(\pi) & \text{ if $\pi \in \C_{B,n+1}^+$ and $0\in \Des_{\Delta}(\pi,\phi(\pi))$}  \\
     -\phi(-\pi) & \text{ if $\pi\in \C_{B,n+1}^-$ and $0\notin \Des_{\Delta}(\pi,-\phi(-\pi))$}\\
     -(-1)  \phi(-\pi) & \text{ if $\pi\in \C_{B,n+1}^-$ and $0\in \Des_{\Delta}(\pi,-\phi(-\pi))$,}\end{cases}\]
     where $(-1)=[-1,2,\ldots,n]\in B_{n}$ is the signed permutation that interchanges $\pm 1$ with $\mp 1$ and fixes all other elements.
\end{definition}

\begin{example}
Recall the signed permutation \[\pi = (-4, -1, 2, 5, -3, -6, 7)=[2,5,-6,-1,-3,7,-4] \in \C_{B,7}^+\]
from \cref{ex:phi}. In that example, we found that
\[\phi(\pi)=(-5)(-1)(2)(4,-3,-6)=[-1,2,-6,-3,-5,4].\]
As $0\in \Des_{\Delta}(\pi,\phi(\pi))$, we have that
\[\Phi(\pi)=(-1)  \phi(\pi)=(-5)(1)(2)(4,-3,-6)=[1,2,-6,-3,-5,4].\]
\end{example}

\begin{example}
Consider the signed permutation 
\[\pi=(3,4,8,-1,5,7,2,-6,-9)=[5,-6,4,8,7,-9,2,-1,3]\in \C_{B,9}^-,\]
which has descents $\{1,4,5,7\}$. For calculating $\Phi(\pi)$, we first must apply Algorithm \ref{algorithm} to 
\[-\pi=(-3,-4,-8,1,-5,-7,-2,6,9)=[5,-6,4,8,7,-9,2,-1,3],\]
which has descents $\{0,2,3,6,8\}$. The algorithm will start with
\[\sigma=(-3,-4,-8)(1,-5,-7,-2)(6)=[-5,1,-4,-8,-7,6,-2,-3],\]
which has descents $\{0,2,3,6,7\}$. The first iteration of the \texttt{for} loop will swap $8$ and $7$, followed by $4$ and $5$, resulting in
\[\sigma=(-3,-5,-7)(1,-4,-8,-2)(6)=[-4,1,-5,-8,-7,6,-3,-2]\]
with descent set $\{0,2,3,6\}$. The algorithm terminates here, resulting in
\[-\phi(-\pi)=(3,5,7)(-1,4,8,2)(-6)=[4,-1,5,8,7,-6,3,2]\]
with descent set $\{1,4,5,7\}$. 
As $0\notin \Des_{\Delta}(\pi,-\phi(-\pi))$, we find that $\Phi(\pi)$ is the permutation $-\phi(-\pi)$ given above. 
\end{example}

We now establish that $\Phi$ preserves descents at all positions in $\{0,1,\ldots,n-1\}$. This proves part (a) of \cref{thm:main_thm}.

\begin{theorem}\label{thm:descents_preserved} 
    Let $\pi\in \C_{B,n+1}$. Then \[\Des(\pi)\cap \{0,1,\ldots,n-1\} =\Des(\Phi(\pi)).\]
\end{theorem}

\begin{proof}
    We first consider $\pi\in \C_{B,n+1}^+$. By \cref{lem:descents_preserved}, we have
    \[\Des_{\Delta}(\pi,\phi(\pi))\cap \{0,1,\ldots,n-1\} \subseteq \{0\}.\]
    When equality holds, \cref{cor:descent0} implies that $\phi(\pi)(1)=-1$, so $\Des(\pi)\cap \{0,1,\ldots,n-1\}=\Des((-1)\phi(\pi))$. The function $\Phi$ makes the appropriate choice so that $\Des(\pi)\cap \{1,2,\ldots,n-1\}=\Des(\Phi(\pi))$, so the result holds for any $\pi\in \C_{B,n+1}^+$. 
    
    Next, consider $\pi\in \C_{B,n+1}^-$. Then $-\pi\in \C_{B,n+1}^+$. From argument above, $\Des(-\pi)\cap \{0,1,\ldots,n-1\}$ is equal to either $\Des(\phi(-\pi))$ or $\Des((-1)\phi(-\pi))$. This implies that $\Des(\pi)\cap \{0,1,\ldots,n-1\}$ is equal to either $\Des(-\phi(-\pi))$ or $\Des(-(-1)\phi(-\pi))$, and $\Phi$ makes the appropriate choice so that $\Des(\pi)\cap \{0,1,\ldots,n-1\}=\Des(\Phi(\pi))$. 
\end{proof}

\section{Mapping signed permutations to cyclic signed permutations}\label{sec:inverses}

In \cref{sec:bijection}, we constructed our descent-preserving function $\Phi:\C_{B,n+1} \to B_n$. We now consider the restrictions $\Phi_{\C_{D,n+1}}=\Phi|_{\C_{D,n+1}}$ and  $\Phi_{\overline{\C_{D,n+1}}}=\Phi|_{\overline{\C_{D,n+1}}}$, which respectively map $\C_{D,n+1}$ and $\overline{\C_{D,n+1}}$ to $B_n$. In this section, we will define inverses $\Psi_{\C_{D,n+1}}$ and $\Psi_{\overline{\C_{D,n+1}}}$ for $\Phi_{\C_{D,n+1}}$ and $\Phi_{\overline{\C_{D,n+1}}}$, which will allow us to conclude that $\Phi_{\C_{D,n+1}}$ and  $\Phi_{\overline{\C_{D,n+1}}}$ are bijections. 

\subsection{A reverse algorithm}\label{sec:inverse_algorithm}

We start with an algorithm that begins with a signed permutation in $B_n$ and outputs a cyclic permutation in $\C_{B,n+1}^+$. This algorithm is intended to reverse the steps in Algorithm \ref{algorithm}.

\begin{definition}
    Define $\psi:B_n\to \C_{B,n+1}^+$ to be the function mapping each $\sigma\in B_n$ to its output from Algorithm \ref{algorithm2}.
\end{definition}

\begin{algorithm}\DontPrintSemicolon
\caption{\algname{Signed to Cyclic Permutation}}\label{algorithm2}
\KwIn{$\sigma\in B_n$}
\KwOut{a permutation in $\C_{B,n+1}^+$}
express $\sigma=\sigma_1\sigma_2\dots\sigma_m$ in canonical cycle notation \\
set $\pi_j=\sigma_j$ for each $j\in \{1,2,\ldots,m\}$ \\
set $\pi_{m+1}=(n+1)$ to be a cycle of length $1$ \\
form $\pi \coloneqq (\pi_{1,1},\ldots,\pi_{1,\ell_1},\pi_{2,1},\ldots,\pi_{2,\ell_2},\ldots ,\pi_{m,1},\ldots \pi_{m,\ell_m},\pi_{m+1,1})$ by concatenating the cycles in $\pi_1,\pi_2,\ldots,\pi_m,\pi_{m+1}$ \label{alg3:setup}\\
\For{\normalfont $j=m-1,m-2,\ldots,1$}
{
    $z\coloneqq $ the rightmost entry of $\pi_j$ \\
    \If{\normalfont $P_{\pi,\sigma}(|z|,|z-1|)$ or $P_{\pi,\sigma}(|z|,|z+1|)$}
    {
        set $\epsilon\in \{-1,1\}$ to be the value such that $P_{\pi,\sigma}(|z|,|z+\epsilon|)$ is \texttt{True} and $\pi(|z+\epsilon|)$ is smallest\\
        \While{\normalfont $P_{\pi,\sigma}(|z|,|z+\epsilon|)$\label{alg3:whileloop1}} 
        {
            $x\coloneqq z$ \\
            $y\coloneqq $ whichever of $z+\epsilon$ or $-(z+\epsilon)$ that appears in $\pi$ \\
            \While{\normalfont $P_{\pi,\sigma}(|x|,|y|)$\label{alg3:whileloop2}}
            {
                respectively replace $x$ and $y$ with $\sgn(x)\cdot |y|$ and $\sgn(y)\cdot |x|$ in $\pi_1,\pi_2,\ldots,\pi_m$ \\
                redefine $\pi$ by concatenating the cycles in $\pi_1,\pi_2,\ldots,\pi_m,\pi_{m+1}$\\
                \If{\normalfont the replacement did not involve the first element of $\pi_j$\label{alg3:if1}}
                {
                $x,y\coloneqq $ the elements respectively preceding the ones replaced
                }
            }
            $z\coloneqq $ the rightmost entry of $\pi_j$
        }
    }
}
\Return $\pi$
\end{algorithm}

We will show that $\phi$ and $\psi$ are inverses of each other, and hence are bijections between $\C_{B,n+1}^+$ and $B_n$. We begin by defining some notation for this section. Fix $\pi\in \C_{B,n+1}^+$, and define $\sigma=\phi(\pi)\in B_n$. For each iteration $j=1,2,\ldots,m$ in the \texttt{for} loop of Algorithm \ref{algorithm} applied to $\pi$, let $\sigma^{(j)}$ be the signed permutation at the start of that loop expressed in the cycle notation from the algorithm, which is canonical by \cref{cor:canonical}. For $j=1,2,\ldots,m$, define $\pi^{(j)}$ to be the cyclic signed permutation formed by concatenating the cycles in $\sigma^{(j)}$ and inserting $n+1$ at the end. As noted in \cref{lem:descents_preserved}, $\sigma^{(m)}$ is the permutation outputted by Algorithm \ref{algorithm}, so the following result is immediate.

\begin{lemma}\label{lem:alg2start}
    The initial permutation in Algorithm \ref{algorithm2} applied to $\sigma$ is $\pi^{(m)}$.
\end{lemma}

We will need to show that Algorithm \ref{algorithm2} reverses the steps in Algorithm \ref{algorithm} by producing $\pi^{(m-1)},\ldots, \pi^{(2)},\pi^{(1)}$ in its \texttt{for} loop. This will require a careful analysis of this loop's relationship with the \texttt{for} loop in Algorithm \ref{algorithm}. As iteration $j$ of the \texttt{for} loop in Algorithm \ref{algorithm} could involve one or more iterations of the \texttt{while} loop in line \ref{line:whileloop1}, we express the permutations that appear between iterations of the loop as
\begin{equation}\label{eq:j-iteration}
    \sigma^{(j,0)},\sigma^{(j,1)},\ldots,\sigma^{(j,h)},
\end{equation}
where $h\geq 0$. Note that $\sigma^{(j,0)}=\sigma^{(j)}$ and $\sigma^{(j,h)}=\sigma^{(j+1)}$. Similar to our notation in \cref{sec:bijection}, we will use $\sigma^{(j,i)}_1\sigma^{(j,i)}_2\dots \sigma_m^{(j,i)}$ to express the individual cycles in $\sigma^{(j,i)}$, and we include a second index in the subscript when referring to a specific entry in a cycle. Finally, define $\pi^{(j,i)}\in \C_{B,n+1}^+$ to be the signed permutation obtained by concatenating the cycles in $\sigma^{(j,i)}$ into one long cycle and inserting $n+1$ at the end. \cref{cor:canonical} implies that each $\sigma^{(j,i)}$ is in canonical cycle notation, so the following lemma follows from a similar argument to the one used in \cref{lem:algstart}.

\begin{lemma}\label{lem:alg2_concatenation}
    Let $0\leq i\leq h$ and $x\in [n-1]$. If $x\in \Des_{\Delta}(\pi^{(j,i)},\sigma^{(j,i)})$, then one of the following must be true:
    \begin{itemize}
        \item $\pm x$ appears as the last element in some cycle of $\sigma^{(j,i)}$, $\pm (x+1)$ appears as a non-last element in some cycle to the right of $\pm x$, and $\pi^{(j,i)}(x)>\pi^{(j,i)}(x+1)=\sigma^{(j,i)}(x+1)>\sigma^{(j,i)}(x)$, or
        \item $\pm (x+1)$ appears as the last element in some cycle of $\sigma^{(j,i)}$, $\pm x$ appears as a non-last element in some cycle to the right of $\pm (x+1)$, and $\pi^{(j,i)}(x+1)>\pi^{(j,i)}(x)=\sigma^{(j,i)}(x)>\sigma^{(j,i)}(x+1)$. 
    \end{itemize}
\end{lemma}

The next lemmas will establish several properties of the $\pi^{(j,i)}$ permutations that can be derived from properties of the $\sigma^{(j,i)}$ permutations. Recall that the order and swap properties hold for all $\sigma^{(j,i)}$ by \cref{thm:order_swap}. We will use these properties without explicitly citing \cref{thm:order_swap}. 

\begin{lemma}\label{lem:sigma_false}
    Let $z$ be the last element of $\sigma_j^{(j,0)}$ and $\epsilon\in \{-1,1\}$. If $P_{\pi,\sigma^{(j,0)}}(|z|,|z+\epsilon|)$ is \texttt{False}, then $P_{\pi^{(j,0)},\sigma}(|z|,|z+\epsilon|)$ is \texttt{False}.
\end{lemma}

\begin{proof}
    We use proof by contradiction. Assume that $P_{\pi,\sigma^{(j,0)}}(|z|,|z+\epsilon|)$ is \texttt{False} and $P_{\pi^{(j,0)},\sigma}(|z|,|z+\epsilon|)$ is \texttt{True}. \cref{lem:descents_preserved} states that $\Des_{\Delta}(\pi,\sigma)\cap [n-1]=\emptyset$, so this combination can only occur if $P_{\pi^{(j,0)},\sigma^{(j,0)}}(|z|,|z+\epsilon|)$ is \texttt{True}. Consequently, we have that $\min\{|z|,|z+\epsilon|\}\in \Des_{\Delta}(\pi^{(j,0)},\sigma^{(j,0)})$, and \cref{lem:alg2_concatenation} implies that 
    \begin{equation}\label{eq:sigma_false}
        \sigma^{(j,0)}_{j+1,1}=\pi^{(j,0)}(|z|)>\pi^{(j,0)}(|z+\epsilon|)=\sigma^{(j,0)}(|z+\epsilon|)>\sigma^{(j,0)}(|z|)=\sigma^{(j,0)}_{j,1}.
    \end{equation}
    Swap property \ref{sp:last} on all signed permutations prior to $\sigma^{(j,0)}$ in Algorithm \ref{algorithm} implies that $\pm z$ has been the last element of cycle $j$ since the beginning of Algorithm \ref{algorithm}, so $\pi(|z|)=\sigma^{(1)}_{j+1,1}$ . However, order property \ref{op:large} for $\sigma$ with \eqref{eq:sigma_false} then implies that $\pi(|z|)=\sigma^{(1)}_{j+1,1}>\pi(|z+\epsilon|)$. This contradicts $P_{\pi,\sigma^{(j,0)}}(|z|,|z+\epsilon|)$ being \texttt{False} and $\sigma^{(j,0)}(|z+\epsilon|)>\sigma^{(j,0)}(|z|)$ from \eqref{eq:sigma_false}. 
\end{proof}

\begin{lemma}\label{lem:phi_iteration}
    Suppose $\pm z$ is the last element in cycle $j$ of $\sigma^{(j,0)}$ and is swapped with $\pm (z+\epsilon)$ in Algorithm \ref{algorithm}, so $\pm (z+i\epsilon)$ is the last element in cycle $j$ of $\sigma^{(j,i)}$ for $0\leq i\leq h$. Then $P_{\pi^{(j,i)},\sigma}(|z+i\epsilon|,|z+(i-1)\epsilon|)$ is \texttt{True} for all $1\leq i\leq h$. 
\end{lemma}

\begin{proof}
    Order property \ref{op:descents} for the respective last elements $\pm z,\pm (z+\epsilon),\ldots,\pm (z+h\epsilon)$ of $\sigma^{(j,0)},\sigma^{(j,1)},\ldots,\sigma^{(j,h)}$ implies that 
    \begin{equation}\label{eq:pi_order}
\pi(|z|)>\pi(|z+\epsilon|)>\pi(|z+2\epsilon|)>\dots > \pi(|z+h\epsilon|),
    \end{equation}
    and \cref{lem:descents_preserved} implies the corresponding order 
    \begin{equation}\label{eq:sigma_order}
\sigma(|z|)>\sigma(|z+\epsilon|)>\sigma(|z+2\epsilon|)>\dots > \sigma(|z+h\epsilon|).
    \end{equation}
    It suffices now to show that $\pi^{(j,i)}(|z+i\epsilon|)>\pi^{(j,i)}(|z+(i-1)\epsilon|)$. 

    We turn our attention to $\sigma^{(j,i-1)}$, where order property \ref{op:descents} implies that the last element $\pm (z+(i-1)\epsilon)\in \sigma_j^{(j,i-1)}$ was swapped with $\pm (z+i\epsilon)$, which appears as a non-last element in some cycle to its right. After swaps occur, $\pm (z+(i-1)\epsilon)$ is not the last element of a cycle in $\sigma^{(j,i)}$, and from this, we conclude \[\sigma^{(j,i)}(|z+(i-1)\epsilon|)=\pi^{(j,i)}(|z+(i-1)\epsilon|).\] 
    Swap property \ref{sp:large} for $\sigma^{(j,i-1)}$ implies that $\sigma^{(j,i-1)}(|z+i\epsilon|)<\sigma_{j+1,1}^{(j,i-1)}$, so after swaps are performed, $\sigma^{(j,i)}(|z+(i-1)\epsilon|)<\sigma_{j+1,1}^{(j,i)}$. Combined, we see that
    \[\pi^{(j,i)}(|z+i\epsilon|)=\sigma_{j+1,1}^{(j,i)}>\sigma^{(j,i)}(|z+(i-1)\epsilon|)=\pi^{(j,i)}(|z+(i-1)\epsilon|).\]
    Comparing with \eqref{eq:sigma_order}, we conclude $P_{\pi^{(j,i)},\sigma}(|z+i\epsilon|,|z+(i-1)\epsilon|)$ is \texttt{True}.
\end{proof}

\begin{lemma}\label{lem:phi_iteration2}
    Suppose $\pm z$ is the last element in cycle $j$ of $\sigma^{(j,0)}$ and is swapped with $\pm (z+\epsilon)$ in Algorithm \ref{algorithm}, so $\pm (z+i\epsilon)$ is the last element in cycle $j$ of $\sigma^{(j,i)}$ for $0\leq i\leq h$. If $P_{\pi^{(j,h)},\sigma}(|z+h\epsilon|,|z+(h+1)\epsilon|)$ is \texttt{True}, then $\pi^{(j,h)}(|z+(h-1)\epsilon)<\pi^{(j,h)}(|z+(h+1)\epsilon|)$. 
\end{lemma}

\begin{proof}
    Since Algorithm \ref{algorithm} terminated iteration $j$ of the \texttt{for} loop at $\sigma^{(j,h)}$, we know that $P_{\pi,\sigma^{(j,h)}}(|z+h\epsilon|,|z+(h+1)\epsilon|)$ is \texttt{False}. \cref{lem:descents_preserved} implies $\Des(\pi)\cap [n-1]=\Des(\sigma)$, and combined with $P_{\pi^{(j,h)},\sigma}(|z+h\epsilon|,|z+(h+1)\epsilon|)$ being \texttt{True}, we conclude that $\min\{|z+h\epsilon|,|z+(h+1)\epsilon|\}\in \Des_{\Delta}(\pi^{(j,h)},\sigma^{(j,h)})$.  As $\pm (z+h\epsilon)$ is the last element of $\sigma_j^{(j,h)}$, \cref{lem:alg2_concatenation} implies that $\pm (z+(h+1)\epsilon)$ must appear as the non-last element in some cycle to the right of $\sigma_j^{(j,h)}$ and
    \[\pi^{(j,h)}(|z+h\epsilon|)>\pi^{(j,h)}(|z+(h+1)\epsilon|)=\sigma^{(j,h)}(|z+(h+1)\epsilon|)>\sigma^{(j,h)}(|z+h\epsilon|).\]
    Since $P_{\pi,\sigma^{(j,h)}}(|z+h\epsilon|,|z+(h+1)\epsilon|)$ is \texttt{False}, we conclude
    \begin{equation}\label{eq:phi_iteration2}
        \pi(|z+(h+1)\epsilon|)>\pi(|z+h\epsilon|).
    \end{equation}
    
    We now turn our attention to $\sigma^{(j,h-1)}$. As $\pm (z+(h+1)\epsilon)$ is not the last element in a cycle of $\sigma_j^{(j,h)}$, swap properties \ref{sp:right} and \ref{sp:last} for $\sigma^{(j,h-1)}$ imply that $\pm (z+(h+1)\epsilon)$ cannot be the last element in a cycle of $\sigma^{(j,h-1)}$. Additionally, order property \ref{op:descents} for $\sigma^{(j,h-1)}$ implies that $\pm (z+h\epsilon)$ was not the last element in a cycle of $\sigma^{(j,h-1)}$. Since $\pm (z+(h+1)\epsilon)$ and $\pm (z+h\epsilon)$ are non-last elements in the cycles of $\sigma^{(j,h-1)}$ containing them, order property \ref{op:descents} for $\sigma^{(j,h-1)}$ implies $\min\{|z+h\epsilon|,|z+(h+1)\epsilon|\}\notin \Des_{\Delta}(\pi,\sigma^{(j,h-1)})$, so combined with \eqref{eq:phi_iteration2}, we see that $\sigma^{(j,h-1)}(|z+(h+1)\epsilon|)>\sigma^{(j,h-1)}(|z+h\epsilon|)$. Then after performing swaps to obtain $\sigma^{(j,h)}$, we have that 
    \begin{equation}\label{eq:hpm1}
        \sigma^{(j,h)}(|z+(h+1)\epsilon|)>\sigma^{(j,h)}(|z+(h-1)\epsilon|).
    \end{equation}Additionally, these swaps result in $\pm (z+(h-1)\epsilon)$ being a non-last element of a cycle in $\sigma^{(j,h)}$. As both $\pm (z+(h-1)\epsilon)$ and $\pm (z+(h+1)\epsilon)$ are non-last elements in the cycles containing them, we conclude that \eqref{eq:hpm1} is equivalent to $\pi^{(j,h)}(|z+(h+1)\epsilon|)>\pi^{(j,h)}(|z+(h-1)\epsilon|).$
\end{proof}

Combining these results, we show that $\phi$ and $\psi$ are inverses. Our approach is to show that the steps within Algorithm \ref{algorithm} are reversed by Algorithm \ref{algorithm2}.

\begin{theorem}\label{thm:inverses_no_n}
    The functions $\phi:\C_{B,n+1}^+\to B_n$ and $\psi:B_n\to \C_{B,n+1}^+$ are inverses. 
\end{theorem}

\begin{proof}
    It suffices to show that $\psi(\phi((\pi))=\pi$ for every $\pi\in \C_{B,n+1}^+$. As $\C_{B,n+1}^+$ and $B_n$ are finite sets of the same size $2^n\cdot n!$, it will follow that $\phi$ and $\psi$ are bijections that must be inverses of one another. By \cref{lem:alg2start}, we know that Algorithm \ref{algorithm2} applied to $\sigma=\phi(\pi)$ will begin with $\pi^{(m)}$. It suffices to show that if $\pi^{(j+1)}$ is the permutation at the start of iteration $j$ of the \texttt{for} loop in Algorithm \ref{algorithm2}, then $\pi^{(j)}$ will be the result at the end, as the result will then follow by induction on $j=m-1,m-2,\ldots,1$.
    
    We first consider the case where no swaps are performed in iteration $j$ of the \texttt{for} loop in Algorithm \ref{algorithm}, so $\sigma^{(j)}=\sigma^{(j+1)}$ and $\pi^{(j)}=\pi^{(j+1)}$. This occurs since $P_{\pi,\sigma^{(j)}}(|z|,|z+\epsilon|)$ is \texttt{False} for each value of $\epsilon\in \{-1,1\}$. Then \cref{lem:sigma_false} implies that $P_{\pi^{(j+1)},\sigma}(|z|,|z+\epsilon|)$ is \texttt{False} for each value of $\epsilon\in \{-1,1\}$, so iteration $j$ of the \texttt{for} loop in Algorithm \ref{algorithm2} will also perform no swaps. Hence, it will begin and terminate with $\pi^{(j+1)}=\pi^{(j)}$. 
    
    Alternatively, suppose swaps do occur. As we have throughout this section, we will use $\sigma^{(j,0)},\sigma^{(j,1)},\ldots,\sigma^{(j,h)}$ for the signed permutations in iteration $j$ of the \texttt{for} loop of Algorithm \ref{algorithm}, and the respective last elements of cycle $j$ will be $\pm z,\pm (z+\epsilon),\ldots,\pm (z+h\epsilon)$. By \cref{lem:phi_iteration,lem:phi_iteration2}, the \texttt{for} loop of Algorithm \ref{algorithm2} will start with $\pi^{(j,h)}$, swap $\pm (z+h\epsilon)$ and $\pm (z+(h-1)\epsilon)$, and then subsequently swap the elements preceding the ones swapped whenever another descent is affected by the prior swap. This reverses exactly the swaps of the \texttt{while} loop in line \ref{line:whileloop2} of Algorithm \ref{algorithm}, and hence results in $\pi^{(j,h-1)}$. Continuing with \cref{lem:phi_iteration} and a similar argument, we conclude that this iteration of the \texttt{for} loop in Algorithm \ref{algorithm2} will then produce $\pi^{(j,h-2)},\ldots,\pi^{(j,0)}$. It suffices now to show that the loop will terminate at $\pi^{(j,0)}$ due to $P_{\pi^{(j,0)},\sigma}(|z|,|z-\epsilon|)$ being \texttt{False}. If $P_{\pi,\sigma^{(j,0)}}(|z|,|z-\epsilon|)$ is \texttt{False}, then this follows from \cref{lem:sigma_false}, so it suffices to consider when $P_{\pi,\sigma^{(j,0)}}(|z|,|z-\epsilon|)$ is \texttt{True}. 
    
    In this case, order property \ref{op:descents} for $\sigma^{(j,0)}$ implies that both $\pm (z+\epsilon)$ and $\pm (z-\epsilon)$ appear as non-last elements in the cycles of $\sigma^{(j,0)}$ containing them. Additionally, Algorithm \ref{algorithm} chose to swap $\pm z$ with $\pm (z+\epsilon)$ when both $P_{\pi,\sigma^{(j,0)}}(|z|,|z-\epsilon|)$ and $P_{\pi,\sigma^{(j,0)}}(|z|,|z+\epsilon|)$ were \texttt{True}, so $\sigma^{(j,0)}(|z+\epsilon|)>\sigma^{(j,0)}(|z-\epsilon|)$. Combined with order property \ref{op:descents} for $\sigma^{(j,0)}$, it must be that
    \begin{equation}\label{eq:terminates}
        \sigma^{(j,0)}(|z|)<\sigma^{(j,0)}(|z-\epsilon|)=\pi^{(j,0)}(|z-\epsilon|)<\sigma^{(j,0)}(|z+\epsilon|)=\pi^{(j,0)}(|z+\epsilon|),
    \end{equation}
    while $\pi(|z|)>\pi(|z+\epsilon|)$ and $\pi(|z|)>\pi(|z-\epsilon|)$. \cref{lem:descents_preserved} then implies $\sigma(|z|)>\sigma(|z+\epsilon|)$ and $\sigma(|z|)>\sigma(|z-\epsilon|)$. The swaps producing $\pi^{(j,0)}$ from $\pi^{(j,1)}$ removed the descent $\min\{|z|,|z+\epsilon|\}\in \Des_{\Delta}(\pi^{(j,1)},\sigma)$, so $\pi^{(j,0)}(|z|)>\pi^{(j,0)}(|z+\epsilon|)$. Combined with \eqref{eq:terminates}, we conclude that $\pi^{(j,0)}(|z|)>\pi^{(j,0)}(|z-\epsilon|)$. As this matches the order of $\sigma(|z|)>\sigma(|z-\epsilon|)$, we see that $P_{\pi^{(j,0)},\sigma}(|z|,|z-\epsilon|)$ is \texttt{False}, so iteration $j$ of the \texttt{for} loop will terminate at $\pi^{(j)}=\pi^{(j,0)}$. 
\end{proof}

\subsection{Inverse functions}\label{sec:inverse_functions}

From \cref{thm:inverses_no_n}, we now have an inverse for $\phi$ given by $\psi$. Our descent-preserving function $\Phi$ was constructed using variations of $\phi$. We will use variations of $\psi$ to construct inverses for $\Phi_{\C_{D,n+1}}$ and $\Phi_{\overline{\C_{D,n+1}}}$.  

First recall from \cref{sec:descent_preserving_map} that $\Phi$ is given by
    \[\Phi(\pi) = \begin{cases} \phi(\pi) & \text{ if $\pi \in \C_{B,n+1}^+$ and $0\notin \Des_{\Delta}(\pi,\phi(\pi))$}  \\
     (-1)  \phi(\pi) & \text{ if $\pi \in \C_{B,n+1}^+$ and $0\in \Des_{\Delta}(\pi,\phi(\pi))$}  \\
     -\phi(-\pi) & \text{ if $\pi\in \C_{B,n+1}^-$ and $0\notin \Des_{\Delta}(\pi,-\phi(-\pi))$}\\
     -(-1)  \phi(-\pi) & \text{ if $\pi\in \C_{B,n+1}^-$ and $0\in \Des_{\Delta}(\pi,-\phi(-\pi))$.}\end{cases}\]
Based on this definition, there are several candidate pre-images for each $\sigma\in B_n$, which we describe in the next lemma.

\begin{lemma}\label{lem:preimages_1}
    For any $\sigma\in B_n$, the following hold.
    \begin{enumerate}[label=(\alph*)]
        \item We have $\psi(\sigma),\psi((-1)\sigma)\in \C_{B,n+1}^+$ and $-\psi(-\sigma),-\psi(-(-1)\sigma)\in \C_{B,n+1}^-$.
        \item The parities of the $\negative$ statistic on $\sigma$, $\psi(\sigma)$, and $-\psi(-(-1)\sigma)$ are the same, and the parities of the $\negative$ statistic on $(-1)\sigma$, $\psi((-1)\sigma)$, and $-\psi(-\sigma)$ are the same.
        \item We have $\Phi^{-1}(\{\sigma,(-1)\sigma\})=\{\psi(\sigma),\psi((-1)\sigma),-\psi(-\sigma),-\psi(-(-1)\sigma)\}$. Furthermore, these four elements belong to different subsets $\C_{B,n+1}^+\cap \C_{D,n+1}$, $\C_{B,n+1}^+\cap \overline{\C_{D,n+1}}$, $\C_{B,n+1}^-\cap \C_{D,n+1}$, and $\C_{B,n+1}^-\cap \overline{\C_{D,n+1}}$ of $\C_{B,n+1}$.  
    \end{enumerate}
\end{lemma}

\begin{proof}
    For (a), observe that in Algorithm \ref{algorithm2}, $n+1$ is inserted in the cycle notation and is not involved in any swaps, as a descent at position $n$ is not considered in the algorithm. This immediately implies that $\psi(\sigma),\psi((-1)\sigma),\psi(-\sigma),\psi(-(-1)\sigma)\in \C_{B,n+1}^+$. Changing the signs of all elements after applying $\psi$ to a signed permutation results in $-\psi(-\sigma),-\psi(-(-1)\sigma)\in \C_{B,n+1}^-$. 

    For (b), observe that inserting $n+1$ and performing the swaps in Algorithm \ref{algorithm2} will not change the number of negatives appearing in the cycle notation, so $\negative(\sigma)=\negative(\psi(\sigma))$. Turning our attention to $-\psi(-(-1)\sigma)$, we use this to find
    \begin{equation*}
        \begin{split}
            \negative(-\psi(-(-1)\sigma)) & =(n+1)-\negative(\psi(-(-1)\sigma)) \\
            & =(n+1)-\negative(-(-1)\sigma) 
            \\
            & =(n+1)-(n-\negative((-1)\sigma)) \\
            & =1+\negative((-1)\sigma),
        \end{split}
    \end{equation*}
    and since $\negative((-1)\sigma)\in \{\negative(\sigma)-1,\negative(\sigma)+1\}$, this has the same parity as $\negative(\sigma)$. The corresponding result for $(-1)\sigma$, $\psi((-1)\sigma)$, and $-\psi(-\sigma)$ follows by replacing $\sigma$ with $(-1)\sigma$ throughout. 
    
    Finally, we consider (c). From \cref{thm:inverses_no_n}, $\phi$ and $\psi$ are inverses. Using this with the definition of $\Phi$, any pre-images of $\{\sigma,(-1)\sigma\}$ in $\C_{B,n+1}^+$ must be in $\{\psi(\sigma),\psi((-1)\sigma)\}$, and any pre-images in $\C_{B,n+1}^-$ must be in $\{-\psi(-\sigma),-\psi(-(-1)\sigma)\}$. 
    Applying $\Phi$ to $\psi(\sigma),\psi((-1)\sigma)\in \C_{B,n+1}^+$ or $-\psi(-\sigma),-\psi(-(-1)\sigma)\in \C_{B,n+1}^-$ must also result in $\sigma$ or $(-1)\sigma$, so combined, we conclude
    \[\Phi^{-1}(\{\sigma,(-1)\sigma\})= \{\psi(\sigma),\psi((-1)\sigma),-\psi(-\sigma),-\psi(-(-1)\sigma)\}.\]
    The remaining claims in (c) now follow from casework based on whether $\negative(\sigma)$ is even or odd, combined with (a) and (b).
\end{proof}

The previous result gives us four candidate elements for where to map each $\sigma\in B_n$ when constructing inverses for $\Phi_{\C_{D,n+1}}$ and $\Phi_{\overline{\C_{D,n+1}}}$. Recall from \cref{cor:descent0} that when $\pi\in \C_{B,n+1}^+$, the multiplication by $(-1)$ in $\Phi(\pi)$ only occurs due to $0\notin \Des(\pi)$ and $\phi(\pi)(1)=-1$. Hence, \cref{lem:preimages_1} is sufficient for determining which element to select when $\sigma(1)\notin \{-1,1\}$. When $\sigma(1)\in \{-1,1\}$, we will need some additional results.

\begin{lemma}\label{lem:preimages_2}
    Let $\sigma\in B_n$, and suppose $\sigma(1)=1$. Define $\pi_1=\psi(\sigma)$, $\pi_2=\psi((-1)\sigma)$, $\pi_3=-\psi(-\sigma)$, and $\pi_4=-\psi(-(-1)\sigma)$. Then $\Phi(\pi_1)=\Phi(\pi_2)=\sigma$ and $\Phi(\pi_3)=\Phi(\pi_4)=(-1)\sigma$.
\end{lemma}

\begin{proof}
    We first consider $\Phi(\pi_1)$, where $\pi_1\in \C_{B,n+1}^+$ by \cref{lem:preimages_1}. If $\Phi(\pi_1)=\phi(\pi_1)$, then \cref{thm:inverses_no_n} implies that $\Phi(\pi_1)=\phi(\psi(\sigma))=\sigma$, so the claim holds. Otherwise, $\Phi(\pi_1)=(-1)\phi(\pi_1)$. By \cref{cor:descent0}, this can only occur if $0\notin \Des(\pi_1)$ and $\phi(\pi_1)(1)=-1$. From  \cref{lem:preimages_1} and the assumption $\sigma(1)=1$, we can conclude $\phi(\pi_1)=(-1)\sigma$, so $\Phi(\pi_1)=(-1)(-1)\sigma=\sigma$. 

    We next consider $\Phi(\pi_2)$, where $\pi_2\in \C_{B,n+1}^+$ by \cref{lem:preimages_1}. If $\Phi(\pi_2)=(-1)\phi(\pi_2)$, then $\Phi(\pi_2)=(-1)(-1)\sigma=\sigma$, so the claim holds. Otherwise, $\Phi(\pi_2)=\phi(\pi_2)=(-1)\sigma$. As $\sigma(1)=1$, we know that $(-1)$ is a cycle in $(-1)\sigma$ and $0\in \Des((-1)\sigma)$. Additionally, $\Phi(\pi_2)=\phi(\pi_2)$ occurs when $0\notin \Des_{\Delta}(\pi_2,\phi(\pi_2))$, so $0\in \Des(\pi_2)$. However, $0\in \Des(\pi_2)$ and $\phi(\pi_2)(1)=(-1)\sigma(1)=-1$ contradicts \cref{cor:nodescent0preserved}. 

    We now consider $\Phi(\pi_3)$, where $\pi_3\in \C_{B,n+1}^-$ by \cref{lem:preimages_1}. In the case $\Phi(\pi_3)=-(-1)\phi(-\pi_3)$, we find \[\Phi(\pi_3)=-(-1)\phi(\psi(-\sigma))=(-1)\sigma,\]
    so the claim holds. Otherwise, $\Phi(\pi_3)=-\phi(-\pi_3)=\sigma$. As $\sigma(1)=1$, we know $0\notin \Des(\sigma)$, so \cref{thm:descents_preserved} implies that $0\notin \Des(\pi_3)$ and $0\in \Des(-\pi_3)$. Additionally, $\Phi(\pi_3)=-\phi(-\pi_3)=\sigma$ implies that $\phi(-\pi_3)=-\sigma$, and since $\sigma(1)=1$, we see that $-\sigma(1)=-1$. However, $0\in \Des(-\pi_3)$ and $\phi(-\pi_3)(1)=-\sigma(1)=-1$ similarly contradicts \cref{cor:nodescent0preserved}.

    Finally, consider $\Phi(\pi_4)$, where $\pi_4\in \C_{B,n+1}^-$ by \cref{lem:preimages_1}. If $\Phi(\pi_4)=-\phi(-\pi_4)$, then we find \[\Phi(\pi_4)=-\phi(\psi(-1)\sigma)=(-1)\sigma,\]
    and the claim holds. Otherwise, \[\Phi(\pi_4)=-(-1)\phi(-\pi_4)=-(-1)\phi(\psi(-(-1)\sigma)=\sigma.\]
    As $\sigma(1)=1$, we know that $0\notin \Des(\sigma)$, and \cref{thm:descents_preserved} implies $0\notin \Des(\pi_4)$. We conclude $0\in \Des(-\pi_4)$ and $0\notin \Des(\phi(-\pi_4))=\Des(-(-1)\sigma)$, which again contradicts \cref{cor:nodescent0preserved}.
\end{proof}

\begin{corollary}\label{cor:preimages_2}
    Let $\sigma\in B_n$, and suppose $\sigma(1)=-1$. Define $\pi_1=\psi(\sigma)$, $\pi_2=\psi((-1)\sigma)$, $\pi_3=-\psi(-\sigma)$, and $\pi_4=-\psi(-(-1)\sigma)$. Then $\Phi(\pi_1)=\Phi(\pi_2)=(-1)\sigma$ and $\Phi(\pi_3)=\Phi(\pi_4)=\sigma$.
\end{corollary}

\begin{proof}
    Define $\sigma'=(-1)\sigma$ so that $\sigma'(1)=1$. The result now follows by applying \cref{lem:preimages_2} to $\sigma'$.
\end{proof}

With these results in mind, we can now construct inverses for $\Phi_{\C_{D,n+1}}$ and $\Psi_{\overline{\C_{D,n+1}}}$. This involves casework with \cref{lem:preimages_1}, \cref{lem:preimages_2}, and \cref{cor:preimages_2} to ensure that the appropriate element from $\Phi^{-1}(\sigma,(-1)\sigma)$ is chosen for each $\sigma\in B_n$.

\begin{definition}
    Define $\Psi_{\C_{D,n+1}}:B_n\to \C_{D,n+1}$ by
    \[ \Psi_{\C_{D,n+1}}(\sigma)=\begin{cases} 
    \psi(\sigma) & \text{ if $\sigma(1)\notin \{-1,1\}$ and $\negative(\sigma)$ is even} \\
    -\psi(-\sigma) & \text{ if $\sigma(1)\notin \{-1,1\}$ and $\negative(\sigma)$ is odd} \\
    \psi(\sigma) & \text{ if $\sigma(1)=1$ and $\negative(\sigma)$ is even} \\
    \psi((-1)  \sigma) & \text{ if $\sigma(1)=1$ and $\negative(\sigma)$ is odd} \\
    -\psi(-(-1)  \sigma) & \text{ if $\sigma(1)=-1$ and $\negative(\sigma)$ is even} \\
    -\psi(-\sigma) & \text{ if $\sigma(1)=-1$ and $\negative(\sigma)$ is odd.}\end{cases}\]
    Define $\Psi_{\overline{\C_{D,n+1}}}:B_n\to \overline{\C_{D,n+1}}$ by
    \[ \Psi_{\overline{\C_{D,n+1}}}(\sigma)=\begin{cases} 
    -\psi(-\sigma) & \text{ if $\sigma(1)\notin \{-1,1\}$ and $\negative(\sigma)$ is even} \\
    \psi(\sigma) & \text{ if $\sigma(1)\notin \{-1,1\}$ and $\negative(\sigma)$ is odd} \\
    \psi((-1)  \sigma) & \text{ if $\sigma(1)=1$ and $\negative(\sigma)$ is even} \\
    \psi(\sigma) & \text{ if $\sigma(1)=1$ and $\negative(\sigma)$ is odd} \\
    -\psi(-\sigma) & \text{ if $\sigma(1)=-1$ and $\negative(\sigma)$ is even} \\
    -\psi(-(-1)  \sigma) & \text{ if $\sigma(1)=-1$ and $\negative(\sigma)$ is odd.}\end{cases}\]
\end{definition}

\begin{theorem}\label{lem:inverses}
    The functions $\Phi_{\C_{D,n+1}}$ and $\Psi_{\C_{D,n+1}}$ are inverses, and the functions $\Phi_{\overline{\C_{D,n+1}}}$ and $\Psi_{\overline{\C_{D,n+1}}}$ are inverses. 
\end{theorem}

\begin{proof}
    We will show this for $\Phi_{\C_{D,n+1}}$ and $\Psi_{\C_{D,n+1}}$. The argument for $\Phi_{\overline{\C_{D,n+1}}}$ and $\Psi_{\overline{\C_{D,n+1}}}$ is similar, so we leave this to the reader. From \cref{lem:preimages_1}, it is straightforward to show that the six cases of $\Psi_{\C_{D,n+1}}$ produce elements in $\C_{D,n+1}$, so $\Psi_{\C_{D,n+1}}$ is well-defined. It suffices now to show $\Phi_{\C_{D,n+1}}(\Psi_{\C_{D,n+1}}(\sigma))=\sigma$ for all $\sigma\in B_n$. As $\C_{D,n+1}$ and $B_n$ are finite sets of the same size $2^n\cdot n!$, it will follow that $\Phi_{\C_{D,n+1}}$ and $\Psi_{\C_{D,n+1}}$ are bijections that must be inverses of one another.

    We first consider when $\sigma(1)\notin \{-1,1\}$, where we will use the fact that $\phi$ and $\psi$ are inverses from \cref{thm:inverses_no_n}. In the case where $\negative(\sigma)$ is even, we have that $\Psi_{\C_{D,n+1}}(\sigma)=\psi(\sigma)\in \C_{B,n+1}^+$ and $\phi(\Psi_{\C_{D,n+1}}(\sigma))=\sigma$. Since $\sigma(1)\notin \{-1,1\}$, \cref{lem:descent0} implies that $\Phi$ will not multiply by $(-1)$, so $\Phi_{\C_{D,n+1}}(\Psi_{\C_{D,n+1}}(\sigma))=\phi(\Psi_{\C_{D,n+1}}(\sigma))=\sigma$. In the case where $\negative(\sigma)$ is odd, we have that $\Psi_{\C_{D,n+1}}=-\psi(-\sigma)\in \C_{B,n+1}^-$ and $-\phi(-\Psi_{\C_{D,n+1}}(\sigma))=\sigma$. \cref{lem:descent0} similarly implies that $\Phi_{\C_{D,n+1}}$ will not multiply by $(-1)$, so $\Phi_{\C_{D,n+1}}(\Psi_{\C_{D,n+1}}(\sigma))=\sigma$.

    We now consider when $\sigma(1)\in \{-1,1\}$. If $\sigma(1)=1$ and $\negative(\sigma)$ is even, then \cref{lem:preimages_2} implies
    \[\Phi_{\C_{D,n+1}}(\Psi_{\C_{D,n+1}}(\sigma))=\Phi(\psi(\sigma))=\sigma.\]
    If $\sigma(1)=1$ and $\negative(\sigma)$ is odd, then \cref{lem:preimages_2} implies 
    \[\Phi_{\C_{D,n+1}}(\Psi_{\C_{D,n+1}}(\sigma))=\Phi(\psi((-1)\sigma))=\sigma.\]
    If $\sigma(1)=-1$ and $\negative(\sigma)$ is even, then \cref{cor:preimages_2} implies 
    \[\Phi_{\C_{D,n+1}}(\Psi_{\C_{D,n+1}}(\sigma))=\Phi(-\psi(-(-1)\sigma))=\sigma.\]
    Finally, if $\sigma(1)=-1$ and $\negative(\sigma)$ is odd, then \cref{cor:preimages_2} implies 
    \[\Phi_{\C_{D,n+1}}(\Psi_{\C_{D,n+1}}(\sigma))=\Phi(-\psi(-\sigma))=\sigma. \qedhere \]
\end{proof}

Our main result, \cref{thm:main_thm}, now follows by combining \cref{thm:descents_preserved} and \cref{lem:inverses}. \cref{thm:main_cor} is immediate from \cref{thm:main_thm}.

\section{Asymptotic normality on cyclic signed permutations}\label{sec:normal}

In this section, we establish \cref{thm:CLT}, which shows that the descent and flag major index statistics on $\C_{B,n},\C_{D,n},$ or $\overline{\C_{D,n}}$ are asymptotically normal. Our general approach will be to apply \cref{thm:main_thm} and \cref{thm:main_cor} to approximate the distributions for the descent and flag major index statistics on $\C_{B,n},\C_{D,n},$ or $\overline{\C_{D,n}}$ with the corresponding distribution on $B_{n-1}$. We then derive asymptotic normality on  $\C_{B,n}$, $\C_{D,n}$, or $\overline{\C_{D,n}}$ using asymptotic normality on $B_{n-1}$.

\begin{theorem}\label{thm:CLTdes}
    Let $(X_n)_{n\geq 1}$ be the random variables corresponding to the descent statistic on $(\C_{B,n})_{n\geq 1}$, $(\C_{D,n})_{n\geq 1}$, or $(\overline{\C_{D,n}})_{n\geq 1}$ with means $(\mu_n)_{n\geq 1}$ and variances $(\sigma_n^2)_{n\geq 1}$. The standardized random variable $(X_n-\mu_n)/\sigma_n$ converges in distribution to the standard normal distribution.
\end{theorem}

\begin{proof}
     Fix $(\Omega_n)_{n\geq 1}$ to be either $(\C_{B,n})_{n\geq 1}$, $(\C_{D,n})_{n\geq 1}$, or $(\overline{\C_{D,n}})_{n\geq 1}$. Define $X_n:\Omega_n\to \mathbb{R}$ to be the random variable corresponding to the descent statistic on $\pi\in \Omega_n$ generated uniformly at random, and define $X_n':\Omega_n\to \mathbb{R}$ to be the corresponding descent statistic on $\Phi(\pi)$, where $\Phi$ is the function from \cref{thm:main_thm}. Using these, express
    \begin{equation}\label{eq:Xn_sum}
            \frac{X_{n}-\mu_{n}}{\sigma_{n}} =\frac{X_n-X'_n}{\sigma_{n}}+\frac{X'_n-\mu_{n-1}}{\sigma_{n-1}}\cdot \frac{\sigma_{n-1}}{\sigma_{n}}+\frac{\mu_{n-1}-\mu_{n}}{\sigma_{n}}.
    \end{equation}
    \cref{thm:main_thm} implies that $X_n-X'_n$ is a random variable that takes values in $\{0,1\}$. When $n$ is sufficiently large, \cref{thm:CLLSY} and \cref{thm:CM} imply $\mu_n=n/2$ and $\sigma_n=\sqrt{(n+1)/12}$. Combining these observations, it is straightforward to verify \eqref{eq:converge} to show that \[\frac{X_n-X'_n}{\sigma_{n}} \xrightarrow{p} 0, \quad  \frac{\sigma_{n-1}}{\sigma_{n}}=\frac{\sqrt{n}}{\sqrt{n+1}} \xrightarrow{p} 1, \quad  \text{ and }\quad  \frac{\mu_{n-1}-\mu_{n}}{\sigma_{n}} = \frac{-1/2}{\sqrt{(n+1)/12}} \xrightarrow{p} 0.\]
    \cref{thm:main_cor} implies that the distribution of $X'_n$ coincides with the distribution of the descent statistic for elements of $B_{n-1}$ generated uniformly at random. Hence, \cref{thm:CM} implies that 
    ${(X'_n-\mu_{n-1})}/{\sigma_{n-1}} \xrightarrow{d} \mathcal{N}(0,1)$. Returning to \eqref{eq:Xn_sum}, we now have that
    \begin{equation}\label{eq:Xn_sum2}
        \frac{X_{n}-\mu_{n}}{\sigma_{n}} =\underbrace{\frac{X_n-X'_n}{\sigma_{n}}}_{\xrightarrow{p}{0}}+\underbrace{\frac{X'_n-\mu_{n-1}}{\sigma_{n-1}}}_{\xrightarrow{d}{\mathcal{N}(0,1)}} \cdot \underbrace{\frac{\sigma_{n-1}}{\sigma_{n}}}_{\xrightarrow{p}{1}}+\underbrace{\frac{\mu_{n-1}-\mu_{n}}{\sigma_{n}}}_{\xrightarrow{p}{0}},
    \end{equation}
    so the result follows from \cref{thm:Slutsky}. 
\end{proof}

\begin{theorem}\label{thm:CLTfmaj}
    Let $(X_n)_{n\geq 1}$ be the random variables corresponding to the flag major index statistic on $(\C_{B,n})_{n\geq 1}$, $(\C_{D,n})_{n\geq 1}$, or $(\overline{\C_{D,n}})_{n\geq 1}$ with means $(\mu_n)_{n\geq 1}$ and variances $(\sigma_n^2)_{n\geq 1}$. The standardized random variable $(X_n-\mu_n)/\sigma_n$ converges in distribution to the standard normal distribution.
\end{theorem}

\begin{proof}
    We proceed as in the proof of \cref{thm:CLTdes} with the flag major index statistic for $X_n$ and $X_n'$ to obtain \eqref{eq:Xn_sum} again. \cref{thm:main_thm}, \cref{thm:CLLSY2}, and \cref{thm:CM2} imply that when $n$ us sufficiently large, we have $0\leq X_n-X_n'\leq 2n+1$, $\mu_n=n^2/4$, and $\sigma_n=\sqrt{4n^3+6n^2-n}/6$. From this, we find that
    \[0\leq \frac{X_n-X_n'}{\sigma_n} \leq \frac{6(2n+1)}{\sqrt{4n^3+6n^2-n}},\]
    \[\frac{\sigma_{n-1}}{\sigma_n} = \frac{\sqrt{4(n-1)^3+6(n-1)^2-(n-1)}}{\sqrt{4n^3+6n^2-n}},\]
    \[\frac{\mu_{n-1}-\mu_n}{\sigma_n} = \frac{-3(2n-1)}{\sqrt{4n^3+6n^2-n}}.\]
    One can similarly show that \eqref{eq:Xn_sum2} holds, and the result again follows from \cref{thm:Slutsky}.
\end{proof}

\section*{Acknowledgements}

The author was partially supported by the University of the South's 2025 Faculty Summer Research Stipend. The author would like to thank Alexander Burstein for proposing a generalization of \cref{thm:elizalde} for signed permutations, which led to this work. The author would also like to thank Zachary Hamaker for helpful conversations.

\bibliographystyle{alphaurl}
\bibliography{Bibliography}

\newcommand{\etalchar}[1]{$^{#1}$}
\begin{thebibliography}{CLL{\etalchar{+}}25}

\bibitem[AR01]{adin_roichman}
Ron~M. Adin and Yuval Roichman.
\newblock The flag major index and group actions on polynomial rings.
\newblock {\em European Journal of Combinatorics}, 22(4):431--446, 2001.

\bibitem[Bar13]{baril}
Jean-Luc Baril.
\newblock Statistics-preserving bijections between classical and cyclic permutations.
\newblock {\em Information Processing Letters}, 113(1):17--22, 2013.

\bibitem[BB05]{BB}
Anders Bj\"orner and Francesco Brenti.
\newblock {\em Combinatorics of Coxeter groups}, volume 231 of {\em Graduate Texts in Mathematics}.
\newblock Springer, New York, 2005.

\bibitem[BD92]{BayerDiaconis}
Dave Bayer and Persi Diaconis.
\newblock Trailing the dovetail shuffle to its lair.
\newblock {\em The Annals of Applied Probability}, 2, 1992.

\bibitem[Bil08]{billingsley}
Patrick Billingsley.
\newblock {\em Probability and Measure}.
\newblock Wiley Series in Probability and Statistics. Wiley, third edition, 2008.

\bibitem[Bre94]{brenti94}
Francesco Brenti.
\newblock q-{E}ulerian polynomials arising from {C}oxeter groups.
\newblock {\em European Journal of Combinatorics}, 15(5):417--441, 1994.

\bibitem[CLL{\etalchar{+}}25]{CLLSY}
Jesse {Campion Loth}, Michael Levet, Kevin Liu, Sheila Sundaram, and Mei Yin.
\newblock Moments of colored permutation statistics on conjugacy classes.
\newblock {\em Annals of Combinatorics}, pages 1--45, 2025.

\bibitem[CM12]{CM2012}
Chak-On Chow and Toufik Mansour.
\newblock Asymptotic probability distributions of some permutation statistics for the wreath product {$C_r \wr \mathfrak{S}_n$}.
\newblock {\em Online Analytic Journal of Combinatorics}, 7:Article \#2, 2012.

\bibitem[DF09]{carries_shuffling}
Persi Diaconis and Jason Fulman.
\newblock Carries, shuffling, and an amazing matrix.
\newblock {\em The American Mathematical Monthly}, 116(9):788--803, 2009.

\bibitem[DMP95]{diaconismcgrathpitman}
Persi Diaconis, Michael McGrath, and Jim Pitman.
\newblock Riffle shuffles, cycles, and descents.
\newblock {\em Combinatorica}, 15(1):11–29, 1995.

\bibitem[Eli11]{elizalde}
Sergi Elizalde.
\newblock Descent sets of cyclic permutations.
\newblock {\em Advances in Applied Mathematics}, 47(4):688--709, 2011.

\bibitem[ET19]{troyka}
Sergi Elizalde and Justin~M. Troyka.
\newblock Exact and asymptotic enumeration of cyclic permutations according to descent set.
\newblock {\em Journal of Combinatorial Theory, Series A}, 165:360--391, 2019.

\bibitem[Ful98]{fulman}
Jason Fulman.
\newblock The distribution of descents in fixed conjugacy classes of the symmetric groups.
\newblock {\em Journal of Combinatorial Theory, Series A}, 84(2):171--180, 1998.

\bibitem[Ful99]{FULMAN1999390}
Jason Fulman.
\newblock Descent identities, {H}essenberg varieties, and the {W}eil conjectures.
\newblock {\em Journal of Combinatorial Theory, Series A}, 87(2):390--397, 1999.

\bibitem[GR93]{gessel}
Ira~M Gessel and Christophe Reutenauer.
\newblock Counting permutations with given cycle structure and descent set.
\newblock {\em Journal of Combinatorial Theory, Series A}, 64(2):189--215, 1993.

\bibitem[Hol97]{holte}
John~M. Holte.
\newblock Carries, combinatorics, and an amazing matrix.
\newblock {\em The American Mathematical Monthly}, 104(2):138--149, 1997.

\bibitem[KL20]{kimlee}
Gene~B. Kim and Sangchul Lee.
\newblock Central limit theorem for descents in conjugacy classes of ${S}_n$.
\newblock {\em Journal of Combinatorial Theory, Series A}, 169:105123, 2020.

\bibitem[LY25]{des_fmaj}
Kevin Liu and Mei Yin.
\newblock Descents and flag major index on conjugacy classes of colored permutation groups without short cycles, 2025.
\newblock \href{https://arxiv.org/abs/2503.02990}{arxiv:2503.02990}.

\bibitem[Pet15]{Petersen2015}
T.~Kyle Petersen.
\newblock {\em Eulerian numbers}.
\newblock Birkh\"{a}user Advanced Texts: Basel Textbooks. Birkh\"{a}user/Springer, New York, 2015.

\bibitem[Poi98]{poirier}
Stéphane Poirier.
\newblock Cycle type and descent set in wreath products.
\newblock {\em Discrete Mathematics}, 180(1):315--343, 1998.

\bibitem[Rei93]{reiner}
Victor Reiner.
\newblock Signed permutation statistics and cycle type.
\newblock {\em European Journal of Combinatorics}, 14(6):569--579, 1993.

\bibitem[Ste94]{steingrim94}
Einar Steingr\'{i}msson.
\newblock Permutation statistics of indexed permutations.
\newblock {\em European Journal of Combinatorics}, 15(2):187--205, 1994.

\end{thebibliography}

\addresseshere

\newpage 

\appendix

\section{Reduction to Elizalde's bijection}\label{appendix:A}

When $\pi\in \C_{S,n+1}$ is inputted into Algorithm \ref{algorithm}, the algorithm can be simplified significantly. In this case, $\pi$ contains only positives in its cycle notation, so this also holds for all signed permutations $\sigma$ encountered in the algorithm. This leads to our first reduction.

\begin{observation}\label{obs:reduction1}
    The absolute values and negative signs in Algorithm \ref{algorithm} can be omitted. Furthermore, $0\notin \Des(\pi)$ and $0\notin \Des(\phi(\pi))$, so $\Phi(\pi)=\phi(\pi)$.
\end{observation}

Order property \ref{op:descents} can also be used to restate the Boolean function $P_{\pi,\sigma}$ used in lines \ref{line:if0} and \ref{line:whileloop1}. This property holds by \cref{thm:order_swap}.

\begin{observation}
    In lines \ref{line:if0} and \ref{line:whileloop1} of Algorithm \ref{algorithm}, $P_{\pi,\sigma}(z,z+\epsilon)$ is \texttt{True} precisely when $\pi(z)>\pi(z+\epsilon)$ and $\sigma(z)<\sigma(z+\epsilon)$.
\end{observation}

Additionally, the \texttt{while} loop in line \ref{line:whileloop2} can be restated. \cref{lem:continued_swaps} describes when it terminates, and the following cases remain when we combine this with \cref{obs:reduction1}. 

\begin{observation}
The \texttt{while} loop on line \ref{line:whileloop2} terminates when the first element of $\sigma_j$ is swapped or the elements preceding the ones swapped differ by more than $1$.
\end{observation}

Combining these observations, we can restate the restriction $\Phi|_{\C_{S,n+1}}$ as the output of Algorithm \ref{algorithm3} given below. This is consistent with the algorithm originally given in \cite[Section 2]{elizalde}.

\begin{algorithm}\DontPrintSemicolon 
\caption{\algname{Cyclic Permutation to Permutation}}\label{algorithm3}
\KwIn{$\pi \in \C_{S,n+1}$}
\KwOut{a permutation in $S_n$}
express $\pi$ in the cycle notation with $n+1$ in the final position \\
$\pi_{i_1},\pi_{i_2},\ldots,\pi_{i_{m+1}}\coloneqq $ the left-to-right maxima in the cycle notation above\\
set $\sigma=\sigma_1\sigma_2\dots \sigma_{m}\coloneqq (\pi_{i_1},\ldots,\pi_{i_2-1})(\pi_{i_2},\ldots,\pi_{i_3-1})\dots (\pi_{i_{m}},\ldots,\pi_{i_{m+1}-1})$ \\
\For{$j=1,2,\ldots,m$} 
{
    $z\coloneqq $ the rightmost entry of $\sigma_j$ \\
    \If{\normalfont $\pi(z)>\pi(z+\epsilon)$ and $\sigma(z)<\sigma(z+\epsilon)$ for some $\epsilon\in \{-1,1\}$}
    {
        set $\epsilon\in \{-1,1\}$ to be the value such that $\pi(z)>\pi(z+\epsilon)$, $\sigma(z)<\sigma(z+\epsilon)$, and $\pi(z+\epsilon)$ is largest\\
        \While{\normalfont $\pi(z)>\pi(z+\epsilon)$ and $\sigma(z)<\sigma(z+\epsilon)$} 
        {
            $x\coloneqq z$ \\
            $y\coloneqq z+\epsilon$\\
            swap $x$ and $y$ in the cycle notation of $\sigma$ \label{line:swap}\\
            \If{\normalfont the swap did not involve the first element of $\sigma_j$ and the elements preceding $x$ and $y$ differ by $1$}
                {
                $x,y\coloneqq $ the elements respectively preceding the ones swapped \\
                continue the algorithm from line \ref{line:swap}
            }
            $z\coloneqq $ the rightmost entry of $\sigma_j$
        }
    }
}
\Return $\sigma$
\end{algorithm}

\newpage 

\section{Colored permutations}\label{appendix:B}

We give a brief description of the colored permutation groups, state analogs of results from this paper, and outline proofs. We refer the reader to \cite{des_fmaj} for a more general description. 

The \emph{colored permutation group} $S_{n,r}$ is the wreath product $\mathbb{Z}_r\wr S_n$, where $\mathbb{Z}_r=\{0,1,\ldots,r-1\}$ is the cyclic group on $r$ elements. Fixing $\zeta = e^{2\pi i/r}$, one can view $S_{n,r}$ as permutations on
\begin{equation}\label{eq:colored_elements}
    \{i\cdot \zeta^c:i\in [n],c\in \mathbb{Z}_r\}
\end{equation}
with the property that whenever $i \mapsto j\cdot \zeta^c$ for $i,j\in [n]$ and $c\in \mathbb{Z}_r$, we have $i\cdot \zeta^{d} \mapsto j\cdot \zeta^{c+d}$ for all $d\in \mathbb{Z}_r$.

We will denote an element of $S_{n,r}$ using a pair $(\omega,\tau)$, where $\omega\in S_n$ and $\tau:[n]\to \mathbb{Z}_r$ is a function. As a permutation on \eqref{eq:colored_elements}, this is defined by $(\omega,\tau)(i\cdot \zeta^c)=\omega(i)\cdot \zeta^{c+\tau(i)}$ for all $i\in [n]$ and $c\in\mathbb{Z}_r$. As the images of the elements in $[n]$ are sufficient for determining a colored permutation, these uniquely determine $(\omega,\tau)$. The \emph{one-line notation} expresses these elements in the form
\[[\omega(1)\cdot \zeta^{\tau(1)},\omega(2)\cdot \zeta^{\tau(2)},\dots,\omega(n)\cdot \zeta^{\tau(n)}].\]

Elements in $S_{n,r}$ can also be expressed in cycle notation. The \emph{two-line cycle notation} for $(\omega,\tau)$ consists of disjoint cycles of the form \[\begin{pmatrix} a_{\ell} &  a_1 & \ldots & a_{\ell-1} \\
a_1\cdot \zeta^{\tau(a_{\ell})} & a_2\cdot \zeta^{\tau(a_1)} & \ldots &   a_{\ell}\cdot \zeta^{\tau(a_{\ell-1})} \end{pmatrix},\]
where $a_1,a_2,\ldots,a_{\ell}\in [n]$. The \emph{(one-line) cycle notation} is obtained by deleting the first line, resulting in
\[\left(a_1\cdot \zeta^{\tau(a_{\ell})},a_2\cdot \zeta^{\tau(a_1)},\ldots, a_{\ell}\cdot \zeta^{\tau(a_{\ell-1})}\right),\]
where we include commas for clarity. A colored permutation $(\omega,\tau)\in S_{n,r}$ is \emph{cyclic} if its cycle notation consists of a single cycle of length $n$, and we use $\C_{S,n,r}$ to denote the set of cyclic permutations in $S_{n,r}$. Additionally, the \emph{color} of $(\omega,\tau)\in S_{n,r}$ is the element $\sum_{i\in [n]} \tau(i) \in \mathbb{Z}_r$. 

We now consider descents. By fixing the ordering 
\[1<2<3<\dots < 1\cdot \zeta <2\cdot \zeta < 3\cdot \zeta <\dots <1\cdot \zeta^{r-1}<2\cdot \zeta^{r-1}<3\cdot \zeta^{r-1}<\ldots,\]
the \emph{descent set} of $(\omega,\tau)\in S_{n,r}$ is defined as
\[\Des(\omega,\tau)=\{i\in [n]: (\omega,\tau)(i)>(\omega,\tau)(i+1)\},\]
where $n+1$ is a fixed point by convention. Alternatively, $i \in \Des(\omega,\tau)$ when either $\tau(i)>\tau(i+1)$, or $\tau(i)=\tau(i+1)$ and $\omega(i)>\omega(i+1)$. From this, the \emph{descent statistic} is defined as $\des(\omega,\tau)=|\Des(\omega,\tau)|$. We emphasize that this differs from the descent set and statistic for a signed permutation in $B_n\cong S_{n,2}$, as the total order used for defining descents is different and $n$ is a possible descent instead of $0$. For this notion of descents, one can establish a generalization of \cref{thm:elizalde}.

\begin{theorem}\label{thm:Snr}
    For any positive integers $n$ and $r$, there exists a function $\Phi: \C_{S,n+1,r}\to S_{n,r}$ with the following properties:
    \begin{enumerate}[label=(\alph*)]
        \item for all $(\omega,\tau)\in \C_{S,n+1,r}$, we have $\Des(\omega,\tau)\cap \{1,2,\ldots,n-1\}=\Des(\Phi(\omega,\tau))$, and
        \item the restriction of $\Phi$ to the elements of $\C_{S,n+1,r}$ with any fixed color is a bijection. 
    \end{enumerate}
\end{theorem}

\begin{proof}
    Let $\phi:\C_{S,n+1}\to S_n$ be the bijection from \cref{thm:elizalde}. Extend this to a map $\Phi$ on $\C_{S,n+1,r}$ by defining $\Phi(\omega,\tau)=(\phi(\omega),\tau|_{[n]})$, where $\tau|_{[n]}$ denotes the restriction of $\tau$. For any $i\in [n-1]$, we consider a descent at position $i$ by focusing on the one-line notations for $(\omega,\tau)$ and $\Phi(\omega,\tau)$. 
    \begin{itemize}
        \item If $\tau(i)\neq \tau(i+1)$, then a descent at position $i$ in $(\omega,\tau)$ and $\Phi(\omega,\tau)$ is determined by the relative order of $\tau(i)=\tau|_{[n]}(i)$ and $\tau(i+1)=\tau|_{[n]}(i+1)$ as elements in $\mathbb{Z}$. Hence, $i\in \Des(\omega,\tau)$ if and only if $i\in \Des(\Phi(\omega,\tau))$. 
        \item If $\tau(i)=\tau(i+1)$, then a descent at position $i$ in $(\omega,\tau)$ and $\Phi(\omega,\tau)$ is determined by whether or not there is a descent at position $i$ in $\omega$ and $\phi(\omega)$. By \cref{thm:elizalde}, $i\in \Des(\omega)$ if and only if $i\in \Des(\phi(\omega))$, so we conclude $i\in \Des(\omega,\tau)$ if and only if $i\in \Des(\Phi(\omega,\tau))$.
    \end{itemize}
    Combined, we see that (a) holds. 
    
    For (b), we fix a color $c\in \mathbb{Z}$. We map each $(\omega',\tau')\in S_{n,r}$ to an element $(\omega,\tau)\in \C_{S,n+1,r}$ as follows:
    \begin{itemize}
        \item define $\omega=\phi^{-1}(\omega')$, and
        \item extend $\tau':[n]\to \mathbb{Z}_r$ to $\tau:[n+1]\to \mathbb{Z}_r$ by defining \[\tau(n+1)=c-\sum_{i\in [n]} \tau'(i)\in \mathbb{Z}_r.\]
    \end{itemize}
    It is straightforward to show that this is the inverse of $\Phi$ on the elements of $\C_{S,n+1,r}$ with color $c$.
\end{proof}

\begin{remark}
    In general, the bijection in the proof of \cref{thm:Snr} does not preserve the descent at position $n$. This descent for $(\omega,\tau)\in \C_{S,n+1,r}$ is determined by the relative order of $(\omega,\tau)(n)$ and $(\omega,\tau)(n+1)$, while this only depends on $\tau(n)$ for $\Phi(\omega,\tau)\in S_{n,r}$. Additionally, $n+1$ is a possible descent for $(\omega,\tau)\in \C_{S,n+1,r}$, while it is not a possible descent for $\Phi(\omega,\tau)\in S_{n,r}$. Combined, we see that $\Des(\omega,\tau)$ and $\Des(\Phi(\omega,\tau))$ can potentially differ at two values, $n$ and $n+1$.
\end{remark}

We next consider the flag major index statistic. For $(\omega,\tau)\in S_{n,r}$, define the following statistics:
\begin{itemize}
    \item the \emph{major index statistic} is $\maj(\omega,\tau)=\sum_{i\in \Des(\omega,\tau)\cap [n-1]} \, i$,
    \item the \emph{color statistic} is $\col(\omega,\tau)=\sum_{i\in [n]} \tau(i)$, and
    \item the \emph{flag major index statistic} is $\fmaj(\omega,\tau)=r\cdot \maj(\omega,\tau)+\col(\omega,\tau)$. 
\end{itemize}
Note that the color statistic chooses representative elements from $\{0,1,\ldots,r-1\}$ for each $\tau(i)$, but adds them as elements in $\mathbb{Z}$. 

One can now establish the following analog of \cref{thm:CLT} for colored permutations. The proof of this result proceeds similarly to the proof of \cref{thm:CLT} given in \cref{sec:normal}. It will use \cref{thm:Snr} in place of \cref{thm:main_thm}, the general $S_{n,r}$ versions of \cref{thm:CM,thm:CM2} given in \cite{CM2012}, and analogs of \cref{thm:CLLSY,thm:CLLSY2} that can be established using the results in \cite{CLLSY}.

\begin{theorem}\label{thm:coloredCLT}
    Let $(X_n)_{n\geq 1}$ be the random variables corresponding to either the descent or flag major index statistic on the cyclic permutations in $S_{n,r}$ with colors in any fixed, nonempty subset of $\mathbb{Z}_r$. Letting $(\mu_n)_{n\geq 1}$ and $(\sigma_n^2)_{n\geq 1}$ denote the respective means and variances, the standardized random variable $(X_n-\mu_n)/\sigma_n$ converges in distribution to the standard normal distribution. 
\end{theorem}

\end{document}